\documentclass[final,5p,times,twocolumn]{elsarticle}

\usepackage{amssymb}
 \usepackage{amsthm}

\journal{}
\usepackage{xspace}
\usepackage{amsmath}
\usepackage{amsfonts}
\usepackage{amssymb}
\usepackage{bm}
\usepackage{subfigure}
\usepackage{xspace}
\usepackage{afterpage}
\usepackage{color}
\usepackage{soul}
\usepackage{cancel}
\usepackage{gensymb}
\usepackage{tikz}
\usetikzlibrary{matrix,backgrounds,shapes}
\usetikzlibrary{positioning}
\usetikzlibrary{fit}
\usetikzlibrary{plotmarks}
\usetikzlibrary{decorations.pathreplacing}
\usepackage{pgfplots}
\usetikzlibrary{shapes}
\usetikzlibrary{positioning,arrows}
\usetikzlibrary{fadings,shapes.arrows,shadows}  
\usetikzlibrary{arrows.meta}
\definecolor{darkorange}{rgb}{1.0, 0.55, 0.0}
\definecolor{do}{rgb}{1.0, 0.55, 0.0}
\definecolor{Royalblue}{rgb}{0.254,0.41,0.88}
\usepackage{xcolor}

\definecolor{darkorange}{rgb}{1.0, 0.55, 0.0}
\definecolor{royalblue}{rgb}{0.254,0.41,0.88}

\newcommand{\Tr}{\ensuremath{^{\mr{T}}}}

\newcommand{\mr}[1]{\ensuremath{\mathrm{#1}}}
\newcommand{\fnc}[1]{\ensuremath{\mathcal{#1}}}
\newcommand{\bfnc}[1]{\ensuremath{\bm{\mathcal{#1}}}}
\newcommand{\mat}[1]{\ensuremath{\mathsf{#1}}}
\newcommand{\etal}[0]{{\em et~al.\@}\xspace}
\newcommand{\eg}[0]{{e.g.\@}\xspace}
\newcommand{\ie}[0]{{i.e.\@}\xspace}
\newcommand{\etc}[0]{{etc.\@}\xspace}

\newcommand{\Eq}[0]{Eq.}
\newcommand{\Th}[0]{\ensuremath{^{\mathrm{th}}}}
\newtheorem{assume}{Assumption}
\newtheorem{thrm}{Theorem}
\DeclareMathOperator{\diag}{diag}

\newcommand{\pL}[0]{\ensuremath{p_{\mathrm{L}}}}
\newcommand{\pH}[0]{\ensuremath{p_{\mathrm{H}}}}
\newcommand{\qL}[0]{\ensuremath{\bm{q}_{\mathrm{L}}}}
\newcommand{\qH}[0]{\ensuremath{\bm{q}_{\mathrm{H}}}}

\newcommand{\rmL}[0]{\mathrm{L}}
\newcommand{\rmH}[0]{\mathrm{H}}

\newcommand{\xm}[1]{\ensuremath{x_{#1}}}
\newcommand{\xil}[1]{\ensuremath{\xi_{#1}}}
\newcommand{\alphal}[1]{\ensuremath{\alpha_{#1}}}
\newcommand{\betal}[1]{\ensuremath{\beta_{#1}}}
\newcommand{\Nl}[1]{\ensuremath{N_{#1}}}
\newcommand{\bxil}[1]{\ensuremath{\bm{\xi}_{#1}}}
\newcommand{\bxili}[2]{\ensuremath{\bm{\xi}_{#1}^{(#2)}}}

\newcommand{\Q}[0]{\ensuremath{\bm{\fnc{Q}}}}
\newcommand{\Jk}[0]{\ensuremath{\fnc{J}_{\kappa}}}
\newcommand{\Qk}[0]{\ensuremath{\bm{\fnc{Q}_{\kappa}}}}
\newcommand{\Jdxildxm}[2]{\ensuremath{\Jk\frac{\partial\xil{#1}}{\partial\xm{#2}}}}
\newcommand{\Fxm}[1]{\ensuremath{\bm{\fnc{F}}_{\xm{#1}}}}
\newcommand{\FxmI}[1]{\ensuremath{\bm{\fnc{F}}_{\xm{#1}}^{I}}}
\newcommand{\FxmV}[1]{\ensuremath{\bm{\fnc{F}}_{\xm{#1}}^{V}}}

\newcommand{\Um}[1]{\ensuremath{\fnc{U}_{#1}}}
\newcommand{\E}[0]{\ensuremath{\fnc{E}}}
\newcommand{\nxm}[1]{\ensuremath{n_{\xm{#1}}}}
\newcommand{\nxil}[1]{\ensuremath{n_{\xil{#1}}}}

\newcommand{\DxiloneD}[1]{\ensuremath{\mat{D}_{\xil{#1}}^{(1D)}}}
\newcommand{\PxiloneD}[1]{\ensuremath{\mat{P}_{\xil{#1}}^{(1D)}}}
\newcommand{\QxiloneD}[1]{\ensuremath{\mat{Q}_{\xil{#1}}^{(1D)}}}
\newcommand{\SxiloneD}[1]{\ensuremath{\mat{S}_{\xil{#1}}^{(1D)}}}
\newcommand{\ExiloneD}[1]{\ensuremath{\mat{E}_{\xil{#1}}^{(1D)}}}
\newcommand{\txilalpha}[1]{\ensuremath{\bm{t}_{\alphal}}}
\newcommand{\txilbeta}[1]{\ensuremath{\bm{t}_{\betal}}}

\newcommand{\PoneD}[0]{\mat{P}^{(1D)}}

\newcommand{\Imat}[1]{\ensuremath{\mat{I}_{#1}}}
\newcommand{\M}[0]{\ensuremath{\mat{P}}}

\newcommand{\Dxil}[1]{\ensuremath{\mat{D}_{\xil{#1}}}}

\newcommand{\Sxil}[1]{\ensuremath{\mat{S}_{\xil{#1}}}}

\newcommand{\Porthol}[1]{\ensuremath{\mat{P}_{\perp\xil{#1}}}}

\newcommand{\ones}[1]{\ensuremath{\bm{1}_{#1}}}
\newcommand{\barones}[1]{\ensuremath{\overline{\bm{1}}_{#1}}}
\newcommand{\wk}[1]{\ensuremath{\bm{w}_{#1}}}

\newcommand{\matJk}[1]{\ensuremath{{\color{orange}\mat{J}}_{#1}}}

\newcommand{\matAlmk}[3]{\ensuremath{\left[{\color{blue}\fnc{J}\frac{\partial\xil{#1}}{\partial\xm{#2}}}\right]_{#3}}}

\newcommand{\ILtoH}[0]{\ensuremath{\mat{I}_{\rmL\mr{to}\rmH}}}
\newcommand{\IHtoL}[0]{\ensuremath{\mat{I}_{\rmH\mr{to}\rmL}}}
\newcommand{\ILtoHoneD}[0]{\ensuremath{\mat{I}_{\rmL\mr{to}\rmH}^{(1D)}}}
\newcommand{\IHtoLoneD}[0]{\ensuremath{\mat{I}_{\rmH\mr{to}\rmL}^{(1D)}}}
\newcommand{\RL}[0]{\ensuremath{\mat{R}_{\rmL}}}
\newcommand{\RH}[0]{\ensuremath{\mat{R}_{\rmH}}}

\newcommand{\Ok}[0]{\ensuremath{\Omega_{\kappa}}}
\newcommand{\pOk}[0]{\ensuremath{\partial\Omega_{\kappa}}}
\newcommand{\Ohat}[0]{\ensuremath{\hat{\Omega}}}
\newcommand{\Ohatk}[0]{\ensuremath{\hat{\Omega}_{\kappa}}}

\newcommand{\Ghat}[0]{\ensuremath{\hat{\Gamma}}}
\newcommand{\pOhatk}[0]{\ensuremath{\partial\hat{\Omega}_{\kappa}}}

\newcommand{\uk}[0]{\ensuremath{\bm{u}_{\kappa}}}

\newcommand{\uL}[0]{\ensuremath{\bm{u}_{\rmL}}}
\newcommand{\uH}[0]{\ensuremath{\bm{u}_{\rmH}}}

\newcommand{\Cij}[2]{\ensuremath{\mat{C}_{#1,#2}}}
\newcommand{\Chatij}[2]{\ensuremath{\hat{\mat{C}}_{#1,#2}}}

\newcommand{\Jtilde}[0]{\ensuremath{{\color{orange}\mat{J}}}}
\newcommand{\utilde}[0]{\ensuremath{\widetilde{\bm{u}}}}
\newcommand{\qtilde}[0]{\ensuremath{\widetilde{\bm{q}}}}

\newcommand{\wtilde}[0]{\ensuremath{\widetilde{\bm{w}}}}
\newcommand{\stilde}[0]{\ensuremath{\widetilde{\bm{s}}}}

\newcommand{\Mtilde}[0]{\ensuremath{\widetilde{\mat{P}}}}
\newcommand{\DtildeH}[1]{\ensuremath{\widetilde{\mat{D}}_{\xil{#1}}}}
\newcommand{\Dtildel}[1]{\ensuremath{\widetilde{\mat{D}}_{\xil{#1}}}}
\newcommand{\Qtildel}[1]{\ensuremath{\widetilde{\mat{Q}}_{\xil{#1}}}}
\newcommand{\Stildel}[1]{\ensuremath{\widetilde{\mat{S}}_{\xil{#1}}}}
\newcommand{\Etildel}[1]{\ensuremath{\widetilde{\mat{E}}_{\xil{#1}}}}
\newcommand{\tildebarmatAlm}[2]{\ensuremath{\left[{\color{blue}\fnc{J}\frac{\partial\xil{#1}}{\partial \xm{#2}}}\right]}}
\newcommand{\tildeone}[0]{\ensuremath{\tilde{\bm{1}}}}

\newcommand{\eonel}[1]{\ensuremath{\bm{e}_{1_{#1}}}}
\newcommand{\eNl}[1]{\ensuremath{\bm{e}_{\Nl{#1}}}}
\newcommand{\bmxi}[1]{\ensuremath{\bm{\xi}^{(#1)}}}
\newcommand{\U}[0]{\ensuremath{\fnc{U}}}

\newcommand{\am}[1]{\ensuremath{a_{#1}}}
\newcommand{\Bm}[1]{\ensuremath{b_{#1}}}
\newcommand{\Chatla}[2]{\ensuremath{\hat{\fnc{C}}_{#1,#2}}}
\newcommand{\Thetaa}[1]{\ensuremath{\Theta_{#1}}}
\newcommand{\matChatla}[2]{\ensuremath{\left[{\color{blue}\Chatla{#1}{#2}}\right]}}
\newcommand{\thetatildea}[1]{\ensuremath{\tilde{\bm{\theta}}_{#1}}}
\newcommand{\thetaa}[1]{\ensuremath{\bm{\theta}_{#1}}}
\newcommand{\IP}[0]{\ensuremath{\bm{I}_{P}}}
\newcommand{\matJG}[1]{\ensuremath{{\color{orange}\mat{J}_{\Ghat^{#1}}}}}

\newtheorem{definition}{Definition}
\newtheorem{remark}{Remark}
\begin{document}

\begin{frontmatter}



\title{Entropy Stable \MakeLowercase{$p$}-Nonconforming Discretizations with the Summation-by-Parts Property for the Compressible Navier--Stokes Equations\tnoteref{t1}}

\tnotetext[t1]{D. Del Rey Fern\'andez was partially supported by a NSERC Postdoctoral Fellowship, 
and gratefully acknowledges this support. Associated with this paper is a NASA TM report that includes many of the proofs which were omitted for brevity.}
\author[nia,nasa]{David C.~Del Rey Fern\'andez\corref{1}\fnref{fn1}}
\ead{dcdelrey@gmail.com}
\author[nasa]{Mark H. Carpenter\fnref{fn2}}
\ead{mark.h.carpenter@nasa.gov}
\author[KAUST]{Lisandro Dalcin\fnref{fn3}}
\ead{dalcinl@gmail.com}
\author[Germany]{Lucas Fredrich\fnref{fn4}}
\ead{lfriedri@math.uni-koeln.de}
\author[Sweden]{Andrew R. Winters\fnref{fn5}}
\ead{andrew.ross.winters@liu.se}
\author[Germany]{Gregor J. Gassner\fnref{fn6}}
\ead{ggassner@math.uni-koeln.de}
\author[KAUST]{Matteo Parsani\fnref{fn7}}
\ead{matteo.parsani@kaust.edu.sa}

\cortext[cor1]{corresponding author}
\fntext[fn1]{Postdoctoral Fellow}
\fntext[fn2]{Senior Research Scientist}
\fntext[fn3]{Research Scientist}
\fntext[fn4]{Data Scientist}
\fntext[fn5]{Research Fellow}
\fntext[fn6]{Professor}
\fntext[fn7]{Assistant Professor}

\address[nia]{National Institute of Aerospace,
  Hampton, Virginia, United States}
\address[nasa]{Computational AeroSciences Branch, NASA Langley Research Center,
  Hampton, Virginia, United States}
\address[Germany]{Mathematical Institute, University of Cologne,
  North Rhine-Westphalia, Germany}
\address[Sweden]{Department of Mathematics (MAI), Link\"{o}ping University, Sweden}
\address[KAUST]{King Abdullah University of Science and Technology (KAUST), 
  Computer Electrical and Mathematical Science and Engineering Division (CEMSE), 
  Extreme Computing Research Center (ECRC), Thuwal, Saudi Arabia}






\begin{abstract}
The entropy conservative, curvilinear, nonconforming, 
$p$-refinement algorithm for hyperbolic conservation laws of Del Rey Fern\'andez~\etal (2019), is extended 
from the compressible Euler equations to the compressible Navier--Stokes equations.
A simple and flexible coupling procedure with planar interpolation operators between adjoining nonconforming elements is used. 
Curvilinear volume metric terms are numerically approximated via a minimization procedure and satisfy 
the discrete geometric conservation law conditions. Distinct curvilinear surface metrics are used on the 
adjoining interfaces to construct the interface coupling terms, thereby localizing the discrete geometric conservation law 
constraints to each individual element. The resulting scheme is entropy conservative/stable, element-wise 
conservative, and freestream preserving.
Viscous interface dissipation operators are developed that retain the entropy stability of the base scheme. The 
accuracy and stability properties of the resulting numerical scheme are shown to be comparable 
to those of the original conforming scheme (achieving $\sim p+1$ convergence)
in the context of 
the viscous shock problem, the Taylor--Green vortex problem at a Reynolds number of
$Re=1,600$, and a subsonic turbulent flow past a sphere at $Re = 2,000$.
\end{abstract}




\end{frontmatter}


\section{Introduction}
For a certain class of partial differential equations (PDEs), such as the linear wave propagation, 
high-order accurate methods are known to be more efficient than low order methods~\cite{Kreiss1972,Swartz1974}. 
Moreover, high-order methods are well suited to exploit the exascale concurrency on next 
generation hardware because their computational kernels are arithmetically dense and typically local (see, for instance,
\cite{hutchinson_isc16,hadri_ccpe_2019}). 
Nevertheless, despite their long history of development, their application 
to nonlinear PDEs for practical application has been limited by robustness issues, particularly 
in the context of $h$-, $p$-, and $r$-refinement algorithms
to better resolve multi-scale physics. Thus, nominally second-order 
accurate discretization 
operators have dominated commercial software development.

In the context of linear and variable coefficient problems, the summation-by-parts (SBP) framework
provides a systematic and discretization-agnostic methodology for the design and analysis of 
arbitrarily high order, provably conservative and stable numerical methods (see the review papers~\cite{Fernandez2014,Svard2014}). 
SBP operators are matrix difference operators that come endowed with a high-order accurate 
approximation to integration by parts (IBP) that telescopes (\ie, results in boundary terms). 
Over each element, the mimetic and telescoping
properties allow a one-to-one match between discrete and continuous stability proofs. The stability of the full spatial discretization 
is achieved by combining the local SBP mechanics with suitable inter-element coupling and boundary conditions 
procedures (e.g., the simultaneous approximation terms (SATs)~\cite{Carpenter1994,Carpenter1999,Nordstrom1999,Nordstrom2001b,Carpenter2010,svard_entropy_stable_euler_wall_2014,Parsani2015,Parsani2015b,dalcin_2019_wall_bc}).

For nonlinear problems, a provable stability analysis and general theory at the discrete level has remained far more elusive. Nonetheless, progress has been 
made, in particular, Tadmor~\cite{Tadmor1987entropy} developed entropy conservative/stable low order finite volume schemes that achieve entropy 
conservation by using two-point flux functions that when contracted with the entropy variables 
result in a telescoping entropy flux. Entropy stability results by adding appropriate dissipation. 
Via the telescoping property, the continuous $L^{2}$ entropy stability analysis is 
mimicked by the semi-discrete stability analysis (for a complete exposition of these ideas, see 
Tadmor~\cite{Tadmor2003}). The essence of Tadmor's approach resulted in the construction 
of numerous entropy stable schemes. For example, Fjordholm~\etal~\cite{Fjordholm2012} 
have constructed high-order accurate essentially 
non-oscillatory schemes and Ray~\etal~\cite{Ray2016} have constructed low-order accurate unstructured finite 
volume discretizations.

Fisher and co-authors extended Tadmor's approach to finite domains by combining the SBP framework 
with Tadmor's two point flux functions~\cite{Fisher2012phd,Fisher2013,FisherCarpenter2013JCPb} to construct entropy 
stable semi-discrete schemes (see also the related work~\cite{Bjorn2009,Bjorn2018}). Since then, 
these ideas have been extended in numerous ways 
(\eg Refs.~\cite{Carpenter2014,Winters2015,Parsani2016,Carpenter2016,Winters2016,Gassner2016b,Winters2017,Chen2017,Derigs2017,Wintermeyer2017,Crean2018,Chan2018,Fernandez2019_staggered}). 
This combination is attractive because it inherits all of the mechanics of SBP schemes for the imposition of boundary 
conditions and inter-element coupling and therefore gives a systematic methodology for discretizing 
problems on complex geometries \cite{Carpenter2014,Parsani2015b}. Furthermore, the resulting discrete stability proofs do not rely on the assumption of exact integration 
(see for example the work of Hughes~\etal~\cite{Hughes1986}). Recently, these ideas have been extended to
achieve fully-discrete explicit entropy stable schemes for the compressible Euler \cite{Friedrich2019} 
and Navier--Stokes equations  \cite{ranocha2019relaxation}.

An alternative approach, developed by Olsson and Oliger~\cite{Olsson1994}, 
Gerritsen and Olsson~\cite{Margot1996} 
and Yee~\etal~\cite{Yee2000} (see also~\cite{Sandham2002,Bjorn2018}), is based upon choosing entropy 
functions that result in a homogeneity property on the compressible Euler fluxes. Using this property, 
the Euler fluxes can be split such that when contracted with the entropy variables stability estimates 
result that are analogous in form to energy estimates obtained for linear PDEs. 

The objective herein, is to construct entropy stable discretizations for the compressible Navier--Stokes
equations (NSE), applicable to high-order accurate $p$-adaptivity.  The work is
a natural extension of the entropy stable $p$-refinement algorithm for the compressible 
Euler equations in Refs.~\cite{Fernandez2019_p_euler,Friedrich2018}. 
The inviscid terms in the NSE are discretized without modifications using an existing 
approach~\cite{Carpenter2014,Carpenter2015,Parsani2015b,Parsani2016}.  The inviscid discretization
requires two sets of metrics: 
1) volume metrics determined numerically by solving a discrete Geometric Conservation Law (GCL) 
constraint, and 2) surface metric terms that are specified.
Viscous terms are discretized using 
a local discontinuous Galerkin (LDG) approach, plus interior penalty (IP) dissipation included 
on interfaces.  The discretization of the viscous terms is a natural extension of the curvilinear 
LDG-IP approaches found in~\cite{Carpenter2014,Carpenter2015,Parsani2015b,Parsani2016}, to include 
non-conforming interfaces, and are entropy stable by construction.

The contributions of the paper are summarized as follows:
\begin{itemize}
\item The LDG-IP approach in Refs.~\cite{Carpenter2014,Carpenter2015,Parsani2015b,Parsani2016} is applied and extended to the curvilinear nonconforming interface problem.
  \begin{itemize} 
    \item 
The viscous operator, written in terms of the entropy variables,
is discretized using an LDG approach with macro-element discretizations.  
A provably stable quadratic form results, 
provided that identical metrics are used in (both) curvilinear transformations. 
    \item Viscous interface dissipation (the IP terms) are a generalization of the inviscid nonconforming
interface dissipation.  A novel average of on-element and interpolated off-element data leads immediately
to an entropy stable IP term.
     \item The resulting scheme is entropy stable and freestream preserving. 
\end{itemize}
\item Numerical evidence is provided demonstrating that the nonconforming algorithm 1) retains 
similar non-linear robustness properties 
    and 2) achieves similar $L^{2}$-norm convergence rates, i.e., $\sim p+1$ (where $p$ is the highest 
degree polynomial exactly differentiated by the differentiation operator) 
as that of the original nonlinearly stable conforming algorithm \cite{Carpenter2014,Parsani2016}
\end{itemize}
The paper is organized as follows. Section~\ref{sec:notation} delineates the notation used herein. 
Section~\ref{sec:discretizationconvectiondiffusion} reviews the nonconforming algorithm for hyperbolic 
conservation laws of Del Rey Fern\'andez~\etal~\cite{Fernandez2019_p_euler,Fernandez2018_TM} in the context of the linear convection 
equation. The extension to viscous terms is demonstrated by considering the convection-diffusion equation. 
The application of the nonconforming algorithm to the viscous components 
of the compressible Navier--Stokes equations is detailed in Section~\ref{sec:NS} while numerical 
experiments are presented in Section~\ref{sec:numNS}. Finally, conclusions are drawn in Section~\ref{sec:conclusions}.
\section{Notation}\label{sec:notation}
The notation used herein is identical to that in~\cite{Fernandez2019_p_euler}; readers familiar with the notation 
can skip to Section~\ref{sec:discretizationconvectiondiffusion}. PDEs are discretized on cubes having Cartesian computational coordinates denoted by 
the triple $(\xil{1},\xil{2},\xil{3})$, where the physical coordinates are denoted by the triple 
$(\xm{1},\xm{2},\xm{3})$. Vectors are represented by lowercase bold font, for example $\bm{u}$, 
while matrices are represented using sans-serif font, for example, $\mat{B}$. Continuous 
functions on a space-time domain are denoted by capital letters in script font.  For example, 
\begin{equation*}
\fnc{U}\left(\xil{1},\xil{2},\xil{3},t\right)\in L^{2}\left(\left[\alphal{1},\betal{1}\right]\times
\left[\alphal{2},\betal{2}\right]\times\left[\alphal{3},\betal{3}\right]\times\left[0,T\right]\right)
\end{equation*}
represents a square integrable function, where $t$ is the temporal coordinate. The restriction of such 
functions onto a set of mesh nodes is denoted by lower case bold font. For example, the restriction of 
$\fnc{U}$ onto a grid of $\Nl{1}\times\Nl{2}\times\Nl{3}$ nodes is given by the vector
\begin{equation*}
\bm{u} = \left[\fnc{U}\left(\bxili{}{1},t\right),\dots,\fnc{U}\left(\bxili{}{N},t\right)\right]\Tr,
\end{equation*}
where, $N$ is the total number of nodes ($N\equiv\Nl{1}\Nl{2}\Nl{3}$) square brackets ($[]$) are used 
to delineate vectors and matrices as well as ranges for variables (the context will make clear which meaning is being used). Moreover, $\bm{\xi}$ is a vector of vectors 
constructed from the three vectors $\bxil{1}$, $\bxil{2}$, and $\bxil{3}$, which are 
vectors of size $\Nl{1}$, $\Nl{2}$, and $\Nl{3}$ and contain the coordinates of the mesh in 
the three computational directions, respectively. Finally, $\bxil{}$ is constructed as 
\begin{equation*}
\bxil{}(3(i-1)+1:3i)\equiv  \bxili{}{i}
\equiv\left[\bxil{1}(i),\bxil{2}(i),\bxil{3}(i)\right]\Tr,
\end{equation*}
where the notation $\bm{u}(i)$ means the $i\Th$ entry of the vector $\bm{u}$ and $\bm{u}(i:j)$ is the subvector 
constructed from $\bm{u}$ using the $i\Th$ through $j\Th$ entries (\ie, Matlab notation is used).

 Oftentimes, monomials are discussed and the following notation is used:
\begin{equation*}
\bxil{l}^{j} \equiv \left[\left(\bxil{l}(1)\right)^{j},\dots,\left(\bxil{l}(\Nl{l})\right)^{j}\right]\Tr,
\end{equation*}
and the 
convention that $\bxil{l}^{j}=\bm{0}$ for $j<0$ is used.

Herein, one-dimensional SBP operators are used to discretize derivatives. 
The definition of a one-dimensional SBP operator in the $\xil{l}$ direction, $l=1,2,3$, 
is~\cite{DCDRF2014,Fernandez2014,Svard2014}
\begin{definition}\label{SBP}
\textbf{Summation-by-parts operator for the first derivative}: A matrix operator, 
$\DxiloneD{l}\in\mathbb{R}^{\Nl{l}\times\Nl{l}}$, is an SBP operator of degree $p$ approximating the derivative 
$\frac{\partial}{\partial \xil{l}}$ on the domain $\xil{l}\in\left[\alphal{l},\betal{l}\right]$ with nodal 
distribution $\bxil{l}$ having $\Nl{l}$ nodes, if 
\begin{enumerate}
\item $\DxiloneD{l}\bxil{l}^{j}=j\bxil{l}^{j-1}$, $j=0,1,\dots,p$;
\item $\DxiloneD{l}\equiv\left(\PxiloneD{l}\right)^{-1}\QxiloneD{l}$, where the norm matrix, 
$\PxiloneD{l}$, is symmetric positive definite;
\item $\QxiloneD{l}\equiv\left(\SxiloneD{l}+\frac{1}{2}\ExiloneD{l}\right)$, 
$\SxiloneD{l}=-\left(\SxiloneD{l}\right)\Tr$, $\ExiloneD{l}=\left(\ExiloneD{l}\right)\Tr$,  
$\ExiloneD{l} = \diag\left(-1,0,\dots,0,1\right)=\eNl{l}\eNl{l}\Tr-\eonel{l}\eonel{l}\Tr$, 
$\eonel{l}\equiv\left[1,0,\dots,0\right]\Tr$, and $\eNl{l}\equiv\left[0,0,\dots,1\right]\Tr$. 
\end{enumerate}
Thus, a degree $p$ SBP operator is 
one that differentiates exactly monomials up to degree $p$.
\end{definition}

In this work, one-dimensional SBP operators are extended to multiple dimensions 
using tensor products ($\otimes$).  The tensor product between the matrices $\mat{A}$ and $\mat{B}$ 
is given as $\mat{A}\otimes\mat{B}$. When referencing individual entries in a matrix the notation $\mat{A}(i,j)$ 
is used, which means 
the $i\Th$ $j\Th$ entry in the matrix $\mat{A}$.

The focus in this paper is exclusively on diagonal-norm SBP operators. Moreover, the same 
one-dimensional SBP operator are used in each direction, each operating on $N$ nodes. 
Specifically, diagonal-norm SBP operators constructed on the Legendre--Gauss--Lobatto (LGL) 
nodes are used, \ie, a discontinuous Galerkin collocated spectral element approach is utilized.

The physical domain $\Omega\subset\mathbb{R}^{3}$, 
with boundary $\Gamma\equiv\partial\Omega$ is partitioned into $K$ non-overlapping 
hexahedral elements. The domain of the $\kappa^{\text{th}}$ element is denoted by $\Ok$ and has 
boundary $\pOk$. Numerically, PDEs are solved in computational 
coordinates, 
where each $\Ok$ is locally transformed to $\Ohatk$, with boundary $\Ghat\equiv\pOhatk$, under the 
following assumption:

\begin{assume}\label{assume:curv}
Each element in physical space is transformed using 
a local and invertible curvilinear coordinate transformation that is compatible at 
shared interfaces, meaning that points in computational space on either side of a 
shared interface  mapped to the same physical location and therefore map back 
to the analogous location in computational space; this is the standard assumption 
that the curvilinear coordinate transformation is water tight.
\end{assume}
\section{A $p$-nonconforming algorithm: Linear convection-diffusion equation}\label{sec:discretizationconvectiondiffusion}
The focus in this paper is on curvilinearly mapped elements with conforming interfaces 
but nonconforming nodal distributions, as occurs, for example, in $p$-refinement. 
The construction of entropy conservative/stable 
discretizations for the compressible Euler equations on Cartesian grids is detailed in Friedrich~\etal~\cite{Friedrich2018}.
The extension to curvilinear coordinates is covered in~\cite{Fernandez2019_p_euler,Fernandez2018_TM} where a
$p$-refinement, curvilinear, interface coupling technique
that maintains 1) accuracy, 2) discrete entropy conservation/stability, and 3) element-wise conservation is presented. 
Herein, the technology presented in~\cite{Fernandez2019_p_euler,Fernandez2018_TM} is extended to the discretization of the viscous 
portion of the compressible Navier--Stokes equations.

\subsection{Scalar convection-diffusion equation: Continuous and semi-discrete analysis}\label{sec:ICASE_Eqn}
Many of the technical challenges in constructing conservative and stable nonconforming discretizations 
for the compressible Navier--Stokes equations are also present in the discretization of the linear convection-diffusion equation. 
For this reason, the interface coupling procedure for the inviscid terms shown in~\cite{Fernandez2019_p_euler,Fernandez2018_TM} as well as their extension to viscous terms 
(the focus of this paper) is presented in the context of this simple linear scalar equation. The linear convection-diffusion equation 
in Cartesian coordinates is given as  
\begin{equation}\label{eq:cartconvectiondiffusion}
  \frac{\partial\U}{\partial t}+\sum\limits_{m=1}^{3}\frac{\partial\left(\am{m}\U\right)}{\partial\xm{m}}=
  \sum\limits_{m=1}^{3}\frac{\partial^{2}(\Bm{m}\U)}{\partial\xm{m}^{2}},
\end{equation}
where $(\am{m}\U)$ are the inviscid fluxes, $\am{m}$ are the (constant) components of the convection speed, 
$\frac{\partial(\Bm{m}\U)}{\partial\xm{m}}$ are the viscous fluxes, and $\Bm{m}$ are the (constant and positive) diffusion coefficients. 
The energy method can be used to determine the stability of~\eqref{eq:cartconvectiondiffusion}. It proceeds by 
multiplying~\eqref{eq:cartconvectiondiffusion} by the solution, ($\U$), 
and after using the chain rule yields
{\small
\begin{equation}\label{eq:cartconvectionenergydiffusion1}
  \frac{1}{2}\frac{\partial\U^{2}}{\partial t}+
  \frac{1}{2}\sum\limits_{m=1}^{3}\frac{\partial\left(a_{m}\U^{2}\right)}{\partial\xm{m}}=
  \sum\limits_{m=1}^{3}\left\{\frac{\partial}{\partial\xm{m}}\left(\U\frac{\partial(\Bm{m}\U)}{\partial\xm{m}}\right)
  -\left(\frac{\partial(\Bm{m}\U)}{\partial\xm{m}}\right)^{2}\right\}.
\end{equation}
}
Integrating over the domain, $\Omega$, using the integration by parts, and 
the Leibniz rule yields
{\small
\begin{equation}\label{eq:cartconvectionenergydiffusion2}
  \begin{split}
  &\frac{\mr{d}}{\mr{d}t}\int_{\Omega}\frac{\U^{2}}{2}\mr{d}\Omega+\\
  &\frac{1}{2}\sum\limits_{m=1}^{3}\left(\oint_{\Gamma}\left\{
    \left(\am{m}\U^{2}\right)
    -2\U\frac{(\Bm{m}\U)}{\partial\xm{m}}
  \right\}\nxm{m}
    \mr{d}\Gamma+2\int_{\Omega}\left(\frac{\partial(\Bm{m}\U)}{\partial\xm{m}}\right)^{2}\mr{d}\Omega\right)=0,
  \end{split}
\end{equation}
}
with $\nxm{m}$ the $m\Th$ component of the outward facing unit normal. \Eq~\eqref{eq:cartconvectionenergydiffusion2} 
demonstrates that the time rate of change of the norm of the solution, 
$\|\U\|^{2}\equiv\int_{\Omega}\U^{2}\mr{d}\Gamma$, depends on surface flux integrals and a viscous dissipation term. 
Therefore, if appropriate boundary conditions are imposed, \Eq~\eqref{eq:cartconvectionenergydiffusion2} leads to an energy estimate 
on the solution and, hence, a proof of stability. 
The SBP framework used in this paper mimics the above energy stability analysis 
in a one-to-one fashion and leads to similar stability statements on the semi-discrete equations.  

Derivatives are approximated using differentiation matrices that 
are defined in computational space and for this purpose, \Eq~\eqref{eq:cartconvectiondiffusion} is transformed using the 
curvilinear coordinate transformation $\xm{m}=\xm{m}\left(\xil{1},\xil{2},\xil{3}\right)$. Thus, 
after expanding the derivatives with the chain rule as 
\begin{equation*}
  \frac{\partial}{\partial\xm{m}}=
  \sum\limits_{l=1}^{3}\frac{\partial\xil{l}}{\partial\xm{m}}\frac{\partial}{\partial\xil{l}},\quad
  \frac{\partial^{2}}{\partial\xm{m}^{2}}=
  \sum\limits_{l,a=1}^{3}\frac{\partial\xil{l}}{\partial\xm{m}}
  \frac{\partial}{\partial\xil{l}}\left(
  \frac{\partial\xil{a}}{\partial \xm{m}}\frac{\partial}{\partial\xil{a}}  
  \right),
\end{equation*} 
and multiplying by the metric Jacobian, ($\Jk$), \eqref{eq:cartconvectiondiffusion} becomes 
{\small
 \begin{equation}\label{eq:convectiondiffusionchain}
  \Jk\frac{\partial\U}{\partial t}
 +\sum\limits_{l,m=1}^{3}\Jk\frac{\partial\xil{l}}{\partial\xm{m}}
  \frac{\partial \left(a_{m}\U\right)}{\partial\xil{l}}=\sum\limits_{l,a,m=1}^{3}
  \Jk\frac{\partial\xil{l}}{\partial\xm{m}}\frac{\partial}{\partial\xil{l}}
  \left(\frac{\partial\xil{a}}{\partial \xm{m}}
  \frac{\partial(\Bm{m}\U)}{\partial\xil{a}}
  \right).
\end{equation}
}
Herein, Eq.~\eqref{eq:convectiondiffusionchain} is referenced as the chain rule form of Eq.~\eqref{eq:cartconvectiondiffusion}. 
Bringing the metric terms, $\Jdxildxm{l}{m}$, inside the derivative and using again the chain rule gives
{\small
 \begin{equation}\label{eq:convectiondiffusionstrong1}
  \begin{split}
  \Jk\frac{\partial\U}{\partial t}+\sum\limits_{l,m=1}^{3}
  \frac{\partial}{\partial\xil{l}}\left(\Jdxildxm{l}{m}a_{m}\U\right)
-&\sum\limits_{l,m=1}^{3}a_{m}\U\frac{\partial}{\partial\xil{l}}\left(\Jdxildxm{l}{m}\right)
 =\\
 \sum\limits_{l,a,m=1}^{3}
  \frac{\partial}{\partial\xil{l}}
  \left(\Jk\frac{\partial\xil{l}}{\partial\xm{m}}\frac{\partial\xil{a}}{\partial \xm{m}}
  \frac{\partial(\Bm{m}\U)}{\partial\xil{a}}\right)
-&\sum\limits_{l,a,m=1}^{3}
\frac{\partial\xil{a}}{\partial \xm{m}}
  \frac{\partial(\Bm{m}\U)}{\partial\xil{a}}
  \frac{\partial}{\partial\xil{l}}\left( \Jk\frac{\partial\xil{l}}{\partial\xm{m}}\right).
  \end{split}
\end{equation}
}
The last terms on the left- and right-hand sides of~\eqref{eq:convectiondiffusionstrong1} is zero via the GCL relations
\begin{equation}\label{eq:GCL}
    \sum\limits_{l=1}^{3}\frac{\partial}{\partial\xil{l}}\left(\Jdxildxm{l}{m}\right)=0,\quad m=1,2,3,
\end{equation}
leading to the strong conservation form of the convection-diffusion equation in curvilinear coordinates:
{\small
 \begin{equation}\label{eq:convectiondiffusionstrong}
  \Jk\frac{\partial\U}{\partial t}+\sum\limits_{l,m=1}^{3}
  \frac{\partial}{\partial\xil{l}}\left(\Jdxildxm{l}{m}a_{m}\U\right)= \sum\limits_{l,a,m=1}^{3}
  \frac{\partial}{\partial\xil{l}}
  \left(\Jk\frac{\partial\xil{l}}{\partial\xm{m}}\frac{\partial\xil{a}}{\partial \xm{m}}
  \frac{\partial(\Bm{m}\U)}{\partial\xil{a}}\right).
\end{equation}
}

Now, consider discretizing Eq.~\eqref{eq:convectiondiffusionstrong} by using 
the following differentiation matrices:
\begin{equation*}
\Dxil{1}\equiv\mat{D}^{(1D)}\otimes\Imat{N}\otimes\Imat{N},\;
\Dxil{2}\equiv\Imat{N}\otimes\mat{D}^{(1D)}\otimes\Imat{N},\;
\Dxil{3}\equiv\Imat{N}\otimes\Imat{N}\otimes\mat{D}^{(1D)},
\end{equation*} 
where $\Imat{N}$ is an $N\times N$ identity matrix. The diagonal 
matrices containing the metric Jacobian and metric terms along their diagonals, respectively, are defined as follows:
\begin{equation*}
  \begin{split}
  &\matJk{\kappa}\equiv\diag\left(\Jk(\bmxi{1}),\dots,\Jk(\bmxi{\Nl{\kappa}})\right),\\
  &\matAlmk{l}{m}{\kappa}\equiv\diag\left(\Jdxildxm{l}{m}(\bmxi{1}),\dots,
  \Jdxildxm{l}{m}(\bmxi{\Nl{\kappa}})\right),
  \end{split}
\end{equation*}
where $\Nl{\kappa}\equiv N^{3}$ is the total number of nodes in element $\kappa$.
Using this nomenclature, the discretization of~\eqref{eq:convectiondiffusionstrong} on the $\kappa\Th$ element reads
{\small
 \begin{equation}\label{eq:convectionstrongdisc}
  \begin{split}
  &\matJk{\kappa}\frac{\mr{d}\uk}{\mr{d}t}+\sum\limits_{l,m=1}^{3}\Dxil{l}\matAlmk{l}{m}{\kappa}\uk=\\
  &\sum\limits_{l,m,a=1}^{3}\Bm{m}\Dxil{l}^{\kappa}\matJk{\kappa}^{-1}\matAlmk{l}{m}{\kappa}\matAlmk{a}{m}{\kappa}\Dxil{a}^{\kappa}\uk
  +\bm{SAT},
  \end{split}
\end{equation}
}
where $\bm{SAT}$ is the vector of the SATs used to impose both boundary conditions and/or inter-element connectivity. 
Unfortunately, the scheme~\eqref{eq:convectionstrongdisc} is not guaranteed to be stable. However, a well-known remedy is to canonically split the inviscid terms into
one half of the inviscid terms in~\eqref{eq:convectiondiffusionchain} 
and one half of the inviscid terms in~\eqref{eq:convectiondiffusionstrong1} 
(see, for instance, \cite{Carpenter2015}), while the viscous terms are treated 
in strong conservation form.  This process leads to
{\small
\begin{equation}\label{eq:convectiondiffusionsplit}
  \begin{split}
  &\Jk\frac{\partial\U}{\partial t}+\frac{1}{2}\sum\limits_{l,m=1}^{3}\left\{
    \frac{\partial}{\partial\xil{l}}\left(\Jdxildxm{l}{m}a_{m}\U\right)+
     \Jdxildxm{l}{m}\frac{\partial}{\partial\xil{l}}\left(a_{m}\U\right)
    \right\}\\&-\frac{1}{2}\sum\limits_{l,m=1}^{3}\left\{
    a_{m}\U\frac{\partial}{\partial\xil{l}}\left(\Jdxildxm{l}{m}\right)\right\}=
\sum\limits_{l,a,m=1}^{3}
  \frac{\partial}{\partial\xil{l}}
  \left(\Jk\frac{\partial\xil{l}}{\partial\xm{m}}\frac{\partial\xil{a}}{\partial \xm{m}}
  \frac{\partial(\Bm{m}\U)}{\partial\xil{a}}\right),
  \end{split}
\end{equation}
}
where the last set of terms on the left-hand side are zero by the GCL conditions~\eqref{eq:GCL}. Then, a 
stable semi-discrete form is constructed in the same manner as the split form~\eqref{eq:convectiondiffusionsplit} by 
discretizing the inviscid portion of~\eqref{eq:convectiondiffusionchain} and~\eqref{eq:convectiondiffusionstrong} using $\Dxil{l}$, $\matJk{\kappa}$, and  
$\matAlmk{l}{m}{\kappa}$, and averaging the results, while the viscous terms are the discretization of the viscous portion of~\eqref{eq:convectiondiffusionstrong}. 
This procedure yields
\begin{equation}\label{eq:convectionsplitdisc}
  \begin{split}
  &\matJk{\kappa}\frac{\mr{d}\uk}{\mr{d} t}+\frac{1}{2}\sum\limits_{l,m=1}^{3}
  a_{m}\left\{\Dxil{l}\matAlmk{l}{m}{\kappa}+\matAlmk{l}{m}{\kappa}\Dxil{l}\right\}\uk
  \\&-\frac{1}{2}\sum\limits_{l,m=1}^{3}\left\{
    a_{m}\diag\left(\uk\right)\Dxil{l}\matAlmk{l}{m}{\kappa}\ones{\kappa}\right\}=\\
    &\sum\limits_{l,m,a=1}^{3}\Bm{m}\Dxil{l}^{\kappa}\matJk{\kappa}^{-1}\matAlmk{l}{m}{\kappa}\matAlmk{a}{m}{\kappa}\Dxil{a}^{\kappa}\uk,
  \end{split}
\end{equation}
where $\ones{\kappa}$ is a vector of ones of the size of the number of nodes on 
the $\kappa\Th$ element (the SATs have been ignored as they are not important for the current analysis). 
As in the continuous case, the semi-discrete form has a set of discrete GCL conditions
\begin{equation}\label{eq:discGCLconvection}
\sum\limits_{l=1}^{3}
    \Dxil{l}\matAlmk{l}{m}{\kappa}\ones{\kappa}=\bm{0}, \quad m = 1,2,3,
\end{equation}
that if satisfied, lead to the following telescoping (provably stable) semi-discrete form
\begin{equation}\label{eq:convectionsplitdisctele}
  \begin{split}
  &\matJk{\kappa}\frac{\mr{d}\uk}{\mr{d} t}+\frac{1}{2}\sum\limits_{l,m=1}^{3}
  a_{m}\left\{\Dxil{l}\matAlmk{l}{m}{\kappa}+\matAlmk{l}{m}{\kappa}\Dxil{l}\right\}\uk=\\
    &\sum\limits_{l,m,a=1}^{3}\Bm{m}\Dxil{l}^{\kappa}\matJk{\kappa}^{-1}\matAlmk{l}{m}{\kappa}\matAlmk{a}{m}{\kappa}\Dxil{a}^{\kappa}\uk.
  \end{split}
\end{equation}
\begin{remark}
  The linear stability of semi-discrete operators for constant coefficient hyperbolic systems, 
  is not preserved by arbitrary design order approximations to the metric terms.  Only approximations to the metric terms that satisfy
   the discrete GCL conditions~\eqref{eq:discGCLconvection} lead to stable semi-discrete forms. In Del Rey Fern\'andez~\etal~\cite{Fernandez2019_p_euler} 
   full details on how to approximate the metric terms for inviscid terms for nonconforming meshes are given; in the remainder of the paper it is assumed 
   that the metric terms used to discretize inviscid terms satisfy~\eqref{eq:discGCLconvection}. 
\end{remark}

\subsection{Scalar convection-diffusion equation and the nonconforming interface}\label{sec:ICASE_Eqn_NonConform}
The nonconforming semi-discrete algorithms are presented in a simplified setting by considering a single interface 
between two adjoining elements as shown in Figure~\ref{fig:non}. The elements share a vertical interface 
and without loss of generality are assumed to have aligned coordinates. The nonconformity is assumed to arise 
from local approximations with differing polynomial degrees (the analysis is equally valid for other contexts 
such as finite difference blocks with conforming faces but differing numbers of nodes). Specifically, the left element has 
polynomial degree $\pL$ (low: subscript/superscript $\rm{L}$) and the right element has polynomial degree $\pH$ 
(high: subscript/superscript $\rm{H}$) where $\pH>\pL$ (see Figure~\ref{fig:non}). 
\begin{figure}[htbp!]
\centering
\begin{tikzpicture}[scale =0.8]
\begin{axis}
[
axis line style={draw=none},
tick style={draw=none},
ticks=none,
ymin=0.25,ymax=0.8,xmin=0.25,xmax=1.0
]
\addplot[mark=none,line width = 0.5mm, color = darkorange, smooth] table {X11.txt};
\addplot[mark=none,line width = 0.5mm, color = darkorange, draw opacity = 0.5, smooth] table {X12.txt};
\addplot[mark=none,line width = 0.5mm, color = darkorange, draw opacity = 0.5, smooth] table {X13.txt};
\addplot[mark=none,line width = 0.5mm, color = darkorange, smooth] table {X14.txt};

\addplot[mark=none,line width = 0.5mm, color = darkorange, smooth] table {Y11.txt};
\addplot[mark=none,line width = 0.5mm, color = darkorange, draw opacity = 0.5, smooth] table {Y12.txt};
\addplot[mark=none,line width = 0.5mm, color = darkorange, draw opacity = 0.5, smooth] table {Y13.txt};
\addplot[mark=none,line width = 0.5mm, color = darkorange, smooth] table {Y14.txt};

\addplot[mark=*,color = darkorange,smooth, mark options={solid}, only marks,mark size=3pt] table {Y14.txt};

\addplot[mark=none,line width = 0.5mm, color = royalblue, smooth] table {X21.txt};
\addplot[mark=none,line width = 0.5mm, color = royalblue, draw opacity = 0.5, smooth] table {X22.txt};
\addplot[mark=none,line width = 0.5mm, color = royalblue, draw opacity = 0.5, smooth] table {X23.txt};
\addplot[mark=none,line width = 0.5mm, color = royalblue, draw opacity = 0.5, smooth] table {X24.txt};
\addplot[mark=none,line width = 0.5mm, color = royalblue, smooth] table {X25.txt};

\addplot[mark=none,line width = 0.5mm, color = royalblue, smooth] table {Y21.txt};
\addplot[mark=none,line width = 0.5mm, color = royalblue, draw opacity = 0.5, smooth] table {Y22.txt};
\addplot[mark=none,line width = 0.5mm, color = royalblue, draw opacity = 0.5, smooth] table {Y23.txt};
\addplot[mark=none,line width = 0.5mm, color = royalblue, draw opacity = 0.5, smooth] table {Y24.txt};
\addplot[mark=none,line width = 0.5mm, color = royalblue, smooth] table {Y25.txt};

\addplot[mark=square*,color = royalblue, mark options={solid}, smooth, only marks] table {Y21.txt};

\node (elemL) at (250,500) [text width=3cm]{\small Low element: $\pL$, $\uL$, $\Dxil{l}^{\rmL}$, +side};
\node (elemH) at (640,500) [text width=3cm]{\small High element: $\pH$, $\uH$, $\Dxil{l}^{\rmH}$, -side};

          \draw[-latex, line width = 0.5mm,] (10,50,0) -- (10,200,0);
          \draw[-latex, line width = 0.5mm,] (10,50,0) -- (160,50,0);
          \draw[-latex, line width = 0.5mm,] (10,50.5,0) -- (85,125,-2);
          \draw (180,45,0) node[anchor=west] {\small$\xil{1}$};
          \draw (15,215,0) node[anchor=south] {\small$\xil{3}$};
          \draw (100,140,-2) node[anchor=south east] {\small$\xil{2}$};

\end{axis}
\end{tikzpicture}
\caption{Two nonconforming elements.}
\label{fig:non}
\end{figure}
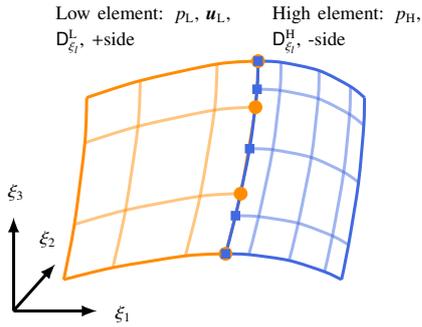

\subsubsection{Review of the inviscid coupling procedure}
First, the energy stable discretization appropriate for hyperbolic conservation laws presented in~\cite{Fernandez2019_p_euler,Fernandez2018_TM} is reviewed. Thus, 
we consider the discretization of the pure convection equation written as 
\begin{equation}\label{eq:convectionsplit}
  \begin{split}
  &\Jk\frac{\partial\U}{\partial t}+\frac{1}{2}\sum\limits_{l,m=1}^{3}\left\{
    \frac{\partial}{\partial\xil{l}}\left(\Jdxildxm{l}{m}a_{m}\U\right)+
     \Jdxildxm{l}{m}\frac{\partial}{\partial\xil{l}}\left(a_{m}\U\right)
    \right\}\\&-\frac{1}{2}\sum\limits_{l,m=1}^{3}\left\{
    a_{m}\U\frac{\partial}{\partial\xil{l}}\left(\Jdxildxm{l}{m}\right)\right\}=0.
  \end{split}
\end{equation}
The basic idea is to construct a macro SBP operator that spans both elements. A naive construction is the following operators 
assembled for the three coordinate directions:
{\small
\begin{equation}\label{eq:naive}
    \DtildeH{1}\equiv\left[
  \begin{array}{cc}
    \Dxil{1}^{\rmL}&\\
    &\Dxil{1}^{\rmH}
  \end{array}  
  \right], \quad
    \Dtildel{2}\equiv\left[
   \begin{array}{cc}
     \Dxil{2}^{\rmL}\\
     &\Dxil{2}^{\rmH} 
   \end{array}
   \right], \quad
    \Dtildel{3}\equiv\left[
   \begin{array}{cc}
     \Dxil{3}^{\rmL}\\
     &\Dxil{3}^{\rmH} 
   \end{array}
   \right].
\end{equation}
}
While the $\Dtildel{2}$ and $\Dtildel{3}$ macro element operators are by construction SBP operators, the $\DtildeH{1}$ is not by construction an 
SBP operator, despite the individual matrices composing $\DtildeH{1}$ being SBP operators.
Moreover, the $\Dtildel{1}$ operator provides no coupling between the two elements. 
To remedy both problems, special interface coupling must be introduced between the two elements.
For this purpose, interpolation operators are needed 
that interpolate information from the $\rmH$ element to the $\rmL$ element and vice versa.  For simplicity, 
the interpolation operators use only tensor product surface information from the adjoining interface surface.

With this background, the general matrix difference operators between the two elements are constructed as 
\begin{equation}
  \Dtildel{l}=\Mtilde^{-1}\Qtildel{l}=\Mtilde^{-1}\left(\Stildel{l}+\frac{1}{2}\Etildel{l}\right).
\end{equation}
The $\Dtildel{l}$, $l=1,2$ matrices satisfy the above decomposition and therefore are SBP. What is necessary is 
to modify the $\Dtildel{1}$ operator. Consider the following modifications:
\begin{equation}\label{eq:marcoD} 
  \begin{split}
  &\Mtilde\equiv\diag\left[
  \begin{array}{cc}  
  \M^{\rmL}\\
  &\M^{\rmH}
  \end{array}
  \right],\\
  &\Stildel{1}\equiv\\
  &
  \left[
  \begin{array}{cc}
    \Sxil{1}^{\rmL}&\tilde{\mat{S}}_{12}\\
   \tilde{\mat{S}}_{21}&\Sxil{1}^{\rmH}
  \end{array}
  \right],\\
  &\tilde{\mat{S}}_{12}\equiv\frac{1}{2}\left(\eNl{\rmL}\eonel{\rmH}\Tr\otimes\PoneD_{\rmL}\IHtoLoneD\otimes\PoneD_{\rmL}\IHtoLoneD\right),\\
  &\tilde{\mat{S}}_{21}\equiv -\frac{1}{2}\left(\eonel{\rmH}\eNl{\rmL}\Tr\otimes\PoneD_{\rmH}\ILtoHoneD\otimes\PoneD_{\rmH}\ILtoHoneD\right),\\
  &\Etildel{1}\equiv\\
  &
  \left[
  \begin{array}{cc}
    -\eonel{\rmL}\eonel{\rmL}\Tr\otimes\PoneD_{\rmL}\otimes\PoneD_{\rmL}\\
    &\eNl{\rmH}\eNl{\rmH}\Tr\otimes\PoneD_{\rmH}\otimes\PoneD_{\rmH}
  \end{array}
  \right],
\end{split}
\end{equation}
and $\IHtoLoneD$ and $\ILtoHoneD$ are one-dimensional interpolation operators from the $\rmH$ element
to the $\rmL$ element and vice versa. 

In order for $\Dtildel{1}$ to be SBP, $\Stildel{1}$ must be skew-symmetric. 
The block-diagonal matrices in $\Stildel{1}$ are 
already skew-symmetric but the off diagonal blocks are not. 
A careful examination reveals that the interpolation operators 
need to be related in the following manner:
\begin{equation}\label{eq:SBPpreserving}
  \IHtoLoneD=\left(\PoneD_{\rmL}\right)^{-1}\left(\ILtoHoneD\right)\Tr\PoneD_{\rmH}.
\end{equation}
Such interpolation operators are denoted as SBP preserving because they lead to a macro element differentiation matrix that is an SBP operator.
The interpolation operator, $\ILtoHoneD$, 
 is constructed to exactly interpolate polynomial of degree $\pL$ and can be easily built as follows:
\begin{equation*}
\ILtoHoneD=\left[\bm{\xi}_{\rmH}^{0},\dots,\bm{\xi}_{\rmH}^{\pL}\right]\left[\bm{\xi}_{\rmL}^{0},\dots,\bm{\xi}_{\rmL}^{\pL}\right]^{-1},
\end{equation*}
where $\bm{\xi}_{\rmL}$ and $\bm{\xi}_{\rmH}$ are the one-dimensional nodal distributions in computational space of the two elements. 
The companion interpolation 
operator $\IHtoLoneD$ is sub-optimal by one degree ($\pL-1$), as a  result of satisfying the necessary 
SBP-preserving property (see Friedrich~\etal~\cite{Friedrich2018} for a thorough discussion). 

The semi-discrete skew-symmetric split operator given in \Eq~\eqref{eq:convectionsplitdisc}, 
discretized using the macro element operators $\Dtildel{l}$, and metric information $\Jtilde$, $\matAlmk{l}{m}{}$,
leads to the following:
\begin{equation}\label{eq:macro}
  \begin{split}
  &\Jtilde\frac{\mr{d}\utilde}{\mr{d}t}+
  \frac{1}{2}\sum\limits_{l,m=1}^{3}a_{m}\left(\Dtildel{l}\tildebarmatAlm{l}{m}+\tildebarmatAlm{l}{m}\Dtildel{l}\right)\utilde\\
  &-\frac{1}{2}\sum\limits_{l,m=1}^{3}a_{m}\diag\left(\utilde\right)\Dtildel{l}\tildebarmatAlm{l}{m}\tildeone
  =\bm{0},
  \end{split}
\end{equation} 
where
\begin{equation}\label{eq:marcmetrics} 
  \begin{split}
  &\utilde\equiv\left[\uL\Tr,\uH\Tr\right]\Tr, \quad
  \Jtilde\equiv\diag\left[
  \begin{array}{cc}  
  \matJk{\rmL}\\
  &\matJk{\rmH}
  \end{array}
  \right], \\
  &\matAlmk{l}{m}{}\equiv
  \left[
  \begin{array}{cc}
  \matAlmk{l}{m}{\rmL}\\
  &\matAlmk{l}{m}{\rmH}
\end{array}  
  \right].
\end{split}
\end{equation}
As for the case in \Eq~\eqref{eq:convectionsplitdisc}, a necessary condition for stability is that the metric
terms satisfy the following discrete GCL conditions:
\begin{equation}\label{eq:discGCLmacro}
  \sum\limits_{l=1}^{3}\Dtildel{l}\tildebarmatAlm{l}{m}\tildeone=\bm{0}.
\end{equation}
Recognizing that $\Dtildel{1}$ is not a tensor product operator, discrete metrics constructed using 
the analytic formalism of Vinokur and Yee~\cite{Vinokur2002a} or Thomas and Lombard~\cite{Thomas1979} 
will not in general satisfy the discrete GCL conditions \Eq~\eqref{eq:discGCLmacro}.  
The only viable alternative is to solve for discrete metrics that directly satisfy such GCL constraints.

\begin{remark}
Note that metric terms are assigned colors; e.g.,  
the time-term Jacobian: $\Jtilde$ or the volume metric terms: $\tildebarmatAlm{l}{m}$.  
Metric terms with common colors form a clique and must be formed consistently.  
For example, the time-term Jacobian and the volume metric Jacobian need not be equivalent.  Another
important clique: the surface metrics, are introduced next.
\end{remark}

The discrete GCL system~\eqref{eq:discGCLmacro} fully couples the approximation of the metrics in 
elements $\rmL$ and $\rmH$.  
Implementation, however, is facilitated by decoupling the GCL computations into individual 
element-wise contributions.
Examination of the skew-symmetric split, curvilinear derivative operator reveals how this is achieved.
The derivative operator may be expressed as
\begin{equation}\label{eq:combined}
  \begin{split}
&\Mtilde\left(\Dtildel{1}\tildebarmatAlm{1}{m}+\tildebarmatAlm{1}{m}\Dtildel{1}\right)=\\
 &\left[
  \begin{array}{cc}
   \mat{A}_{11}&
\mat{A}_{12}\\
-\mat{A}_{12}\Tr
      &\mat{A}_{22}
  \end{array}
  \right]
  +\left(\Etildel{1}\tildebarmatAlm{1}{m}+\tildebarmatAlm{1}{m}\Etildel{1}\right),\\\\
  &\mat{A}_{11}\equiv \left\{\Sxil{1}^{\rmL}\matAlmk{1}{m}{\rmL}+\matAlmk{1}{m}{\rmL}\Sxil{1}^{\rmL}\right\},\\
  &\mat{A}_{12}=  \frac{1}{2}\left\{\begin{array}{l}
    \colorbox{yellow}{\matAlmk{1}{m}{\rmL}}\left(\eNl{\rmL}\eonel{\rmH}\Tr\otimes\PoneD_{\rmL}\IHtoLoneD\otimes\PoneD_{\rmL}\IHtoLoneD\right)\\
    +\left(\eNl{\rmL}\eonel{\rmH}\Tr\otimes\PoneD_{\rmL}\IHtoLoneD\otimes\PoneD_{\rmL}\IHtoLoneD\right)\colorbox{red}{\matAlmk{1}{m}{\rmH}}\end{array}\right\},\\
    &\mat{A}_{22}\equiv\left\{\Sxil{1}^{\rmH}\matAlmk{1}{m}{\rmH}+\matAlmk{1}{m}{\rmH}\Sxil{1}^{\rmH}\right\},
  \end{split}
\end{equation}
with inter-element coupling appearing in the off-diagonal blocks $\mat{A}_{12}$ and $-\mat{A}_{12}\Tr$.  
Replacing the highlighted off-diagonal metric terms in \eqref{eq:combined} with known metric data,
decouples the GCL computation into two element-wise computations. 
The off-diagonal metric data then become forcing terms for individual GCL computations 
in each element.  Note that the ``surface metrics'' appearing in the off-diagonal blocks 
need not be equivalent to those used for the volume metrics that appear on the surfaces.

Construction details for the volume and surface metric terms appear
in~\cite{Fernandez2019_p_euler,Fernandez2018_TM}.  
Careful specification of these terms is essential
when developing (non)linearly stable discretizations for hyperbolic equations in curvilinear coordinates.
The relevant steps are summarized as follows:
\begin{itemize}
\item The highlighted surface metric terms are specified using analytic metrics, resulting in two decoupled GCL computations.
\item Each discrete GCL system is highly underdetermined  and is solved using 
an optimization approach that minimizes the difference between the numerical and analytic volume metrics
\end{itemize}
In contrast, the viscous terms need only use consistent metrics. 
Further remarks are included in Section~\ref{sec:NSNonConform}.

\subsubsection{Extension to the convection-diffusion equation}\label{sec:diffusion}
In this section, with the inviscid terms appropriately discretized, the extension of these ideas to the viscous terms 
is detailed. To make the presentation easier, and to match what will later be done for the compressible Navier--Stokes equations, 
the inviscid and IP terms are lumped into the terms $\fnc{I}_{nv}$ and $\fnc{I}_{P}$, respectively, 
while the viscous terms are simplified. Thus, Eq.~\eqref{eq:convectiondiffusionsplit} reduces to 
\begin{equation}\label{eq:diffusionsplit}
  \begin{split}
  &\Jk\frac{\partial\U}{\partial t}+\fnc{I}_{nv}= \sum\limits_{l,a=1}^{3}
  \frac{\partial}{\partial\xil{l}} \left(\Chatla{l}{a} \Thetaa{a}\right)+\fnc{I}_{P},  \\
  &\Chatla{l}{a}\equiv\sum\limits_{m=1}^{3}\Jk\frac{\partial\xil{l}}{\partial\xm{m}}\frac{\partial\xil{a}}{\partial\xm{m}}\Bm{m},\quad 
  \Thetaa{a}\equiv\frac{\partial\U}{\partial\xil{a}}.
  \end{split}
\end{equation}
 A local discontinuous Galerkin (LDG) and interior penalty approach (IP) approach is used (see references~\cite{Carpenter2014,Carpenter2015,Parsani2015b,Parsani2016}). The IP term is discussed in detail in Section~\ref{sec:NSIP}. 
 In the LDG approach, the discretization of the viscous terms in \Eq~\eqref{eq:diffusionsplit} proceeds in two steps. First, the gradients $\Thetaa{a}$ are discretized, 
 then the derivatives of the viscous fluxes are discretized. Notice that all of the metric terms are contained in the $\Chatla{l}{a}$ term and therefore the critical 
 technology required for stability is to use the SBP preserving macro element previously presented. Thus, the discretization reads
 \begin{equation}\label{eq:diffusionsplitdisc}
  \frac{\mr{d}\utilde}{\mr{d}t}+\bm{I}_{nv}=\sum\limits_{l,a=1}^{3}\Dtildel{l}\matChatla{l}{a}\thetatildea{a}\:+\: \IP,\quad
  \thetatildea{a}=\Dtildel{a}\utilde,
 \end{equation}
 where the inviscid terms are contained in $\bm{I}_{nv}$. The next theorem demonstrates that the proposed discretization of the viscous terms both telescopes the 
 viscous fluxes to the boundary and adds a dissipative term (thus mimicking the continuous energy analysis and also resulting 
 in a provably stable discretization modulo appropriate boundary SATs).
 \begin{thrm}
  Assume that the inviscid terms, $\bm{I}_{nv}$, are discretized as in Section~\ref{sec:ICASE_Eqn_NonConform} and neglect the $\IP$ term from the analysis. 
  Then the discretization~\eqref{eq:diffusionsplitdisc} 
  leads to a telescoping form that under the assumption of appropriate boundary SATs is provably stable. 
 \end{thrm}
 \begin{proof}
  The semi-discrete analysis proceeds as in the continuous case (the inviscid terms are completely dropped as they are assumed to be correctly constructed). 
  Multiplying \Eq~\eqref{eq:diffusionsplit} by $\utilde\Tr\Mtilde$ gives 
 \begin{equation}\label{eq:diffusionsplitdisceng1}
  \utilde\Tr\Mtilde\frac{\mr{d}\utilde}{\mr{d}t}=\sum\limits_{l,a=1}^{3}\utilde\Tr\Mtilde\Dtildel{l}\matChatla{l}{a}\Dtildel{a}\utilde.
 \end{equation}
 Using the SBP property $\Qtildel{l}=-\Qtildel{l}\Tr+\Etildel{l}$, \eqref{eq:diffusionsplitdisceng1} reduces to
 \begin{equation}\label{eq:diffusionsplitdisceng2}
  \utilde\Tr\Mtilde\frac{\mr{d}\utilde}{\mr{d}t}=\sum\limits_{l,a=1}^{3}\utilde\Tr\Etildel{l}\matChatla{l}{a}\Dtildel{a}\utilde
-\sum\limits_{l,a=1}^{3}\left(\Dtildel{l}\utilde\right)\Tr\Mtilde\matChatla{l}{a}\Dtildel{a}\utilde.
 \end{equation}
 The first term on the right-hand side is the discrete equivalent of 
 \begin{equation*}
  \oint_{\Ghat}\sum\limits_{l,a=1}^{3}\left(\Chatla{l}{a}\U\:\frac{\partial\U}{\partial\xil{a}}\right)\nxil{l}\mr{d}\Ghat
 \end{equation*}
 and appropriate SATs need to be imposed to obtain an energy estimate. The second term on the right-hand side 
 is the discrete equivalent to 
 \begin{equation*}
  -\oint_{\Ohat}\sum\limits_{l,a=1}^{3}\Chatla{l}{a}\left(\frac{\partial\U}{\partial\xil{l}}\right)\left(\frac{\partial\U}{\partial\xil{a}}\right)\mr{d}\Ohat,
 \end{equation*}
 and is negative semi-definite.
 \end{proof}
\section{Application to the compressible Navier--Stokes equations}\label{sec:NS}
In this section, first the entropy stability of the continuous compressible Navier--stokes equations is reviewed. Then, 
the nonconforming algorithm for the diffusion equation is applied to the viscous terms 
of the compressible Navier--Stokes equations. In order to obtain an entropy stable formulation, the 
viscous terms are recast in terms of entropy variables, thereby leading to a quadratic form that 
matches in form the terms in the diffusion equation. This allows a direct application of the algorithm 
for the diffusion equation to the compressible viscous terms and entropy stability is proven in an analogous fashion to 
linear stability.
\subsection{Review of the continuous entropy analysis}
The entropy stable algorithm discretizes the skew-symmetric form (in terms of the metric terms) of the compressible Navier--Stokes 
equations, which are given as
\begin{equation}\label{eq:NSCCS1}
\begin{split}
&\Jk\frac{\partial\Qk}{\partial t}+\sum\limits_{l,m=1}^{3}\frac{1}{2}
\frac{\partial }{\partial \xil{l}}\left(\Jdxildxm{l}{m}\FxmI{m}\right)
+\frac{1}{2}\Jdxildxm{l}{m}\frac{\partial \FxmI{m}}{\partial \xil{l}}
=\\&\sum\limits_{l,m=1}^{3}\frac{\partial}{\partial\xil{l}}\left(\Jdxildxm{l}{m}\FxmV{m}\right).
\end{split}
\end{equation}
The vectors $\Q$, $\FxmI{m}$, and $\FxmV{m}$ are the conserved variables, the inviscid fluxes, and the
viscous fluxes, respectively.
The vector of conserved variables is given by 
\begin{equation*}
\Q = \left[\rho,\rho\Um{1},\rho\Um{2},\rho\Um{3},\rho\E\right]\Tr,
\end{equation*}
where $\rho$ denotes the density, $\bm{\fnc{U}} = \left[\Um{1},\Um{2},\Um{3}\right]\Tr$ is the velocity 
vector, and $\E$ is the specific total energy. The inviscid fluxes are given as
\begin{equation*}
  \begin{split}
\FxmI{m} = &\left[\rho\Um{m},\rho\Um{m}\Um{1}+\delta_{m,1}\fnc{P},\rho\Um{m}\Um{2}+\delta_{m,2}\fnc{P},\right.\\
&\left.\rho\Um{m}\Um{3}+\delta_{m,3}\fnc{P},\rho\Um{m}\fnc{H}\right]\Tr,
  \end{split}
\end{equation*}
where $\fnc{P}$ is the pressure, $\fnc{H}$ is the specific total enthalpy and $\delta_{i,j}$ is the 
Kronecker delta.

The required constituent relations are
\begin{equation*}
\fnc{H} = c_{\fnc{P}}\fnc{T}+\frac{1}{2}\bm{\fnc{U}}\Tr\bm{\fnc{U}},\quad \fnc{P} = \rho R \fnc{T},\quad R = \frac{R_{u}}{M_{w}},
\end{equation*}
where $\fnc{T}$ is the temperature, $R_{u}$ is the universal gas constant, $M_{w}$ is the molecular weight of the gas, 
and $c_{\fnc{P}}$ is the specific heat capacity at constant pressure. Finally, the specific thermodynamic entropy is given as 
\begin{equation*}
s=\frac{R}{\gamma-1}\log\left(\frac{\fnc{T}}{\fnc{T}_{\infty}}\right)-R\log\left(\frac{\rho}{\rho_{\infty}}\right),\quad \gamma=\frac{c_{p}}{c_{p}-R},
\end{equation*}
where $\fnc{T}_{\infty}$ and $\rho_{\infty}$ are the reference temperature and density, respectively
(the stipulated convention has been broken here and $s$ has been used rather than $\fnc{S}$ for reasons that will be clear next). 

The viscous fluxes, $\FxmV{m}$, is given as
\begin{equation}\label{eq:Fv}
\FxmV{m}=\left[0,\tau_{1,m},\tau_{2,m},\tau_{3,m},
\sum\limits_{i=1}^{3}\tau_{i,m}\fnc{U}_{i}-\kappa\frac{\partial \fnc{T}}{\partial\xm{m}}\right]\Tr.
\end{equation} 
The viscous stresses are defined as
\begin{equation}\label{eq:tau}
\tau_{i,j} = \mu\left(\frac{\partial\fnc{U}_{i}}{\partial x_{j}}+\frac{\partial\fnc{U}_{j}}{\partial x_{i}}
-\delta_{i,j}\frac{2}{3}\sum\limits_{n=1}^{3}\frac{\partial\fnc{U}_{n}}{\partial x_{n}}\right),
\end{equation}
where $\mu(T)$ is the dynamic viscosity and $\kappa(T)$ is the thermal conductivity (not to be confused with the choice of 
parameter for element numbering). 

The compressible Navier--Stokes equations given in \eqref{eq:NSCCS1} have
a convex extension, that when integrated over the physical domain, $\Omega$, 
depends only on the boundary data and negative semi-definite dissipation terms.
This convex extension depends on an entropy function, $\fnc{S}$, 
that is constructed from the thermodynamic entropy as
\[
\fnc{S}=-\rho s,
\]
and provides a mechanism for proving stability in the $L^{2}$ norm.  
The entropy variables: $\bfnc{W}$,
are an alternate variable set related to the conservative variables via a one-to-one mapping.  They
are defined in terms of the entropy function $\fnc{S}$ by the relation
$\bfnc{W}\Tr=\partial\fnc{S}/\partial\bfnc{Q}$.  The entropy variables are used extensively in the
stability proofs to follow.  They also have the property that they 
simultaneously symmetrize
the inviscid and the viscous flux Jacobians in all spatial directions.
Greater details on continuous entropy analysis is available 
elsewhere \cite{dafermos-book-2010,Svard2015,Carpenter2015}.

The proof of entropy stability for the viscous terms in the compressible Navier--Stokes 
equations \eqref{eq:NSCCS1} is most readily demonstrated by 
exploiting the symmetrizing properties of the 
entropy variables: $\bfnc{W}\equiv\partial\fnc{S}/\partial\bfnc{Q}$.
Recasting the viscous fluxes in the entropy variables results in 
\begin{equation}\label{eq:Fxment}
  \FxmV{m}=\sum\limits_{j=1}^{3}\Cij{m}{j}\frac{\partial\bfnc{W}}{\partial x_{j}},
\end{equation}
with the flux Jacobian matrices satisfying $\Cij{m}{j} \:=\: {(\Cij{j}{m})}\Tr$.
Thus, transforming~\eqref{eq:Fxment} to curvilinear coordinates and substituting the result into~\eqref{eq:NSCCS1}, 
results in the form of the Navier--Stokes equations which is discretized:
\begin{equation}\label{eq:NSCCS}
\begin{split}
&\Jk\frac{\partial\Qk}{\partial t}+\sum\limits_{l,m=1}^{3}
\frac{1}{2}\frac{\partial }{\partial \xil{l}}\left(\Jdxildxm{l}{m}\Fxm{l}\right)
+\frac{1}{2}\Jdxildxm{l}{m}\frac{\partial \Fxm{m}}{\partial \xil{l}}
=\\&\sum\limits_{l,a=1}^{3}\frac{\partial}{\partial\xil{l}}\left(\Chatij{l}{a}\frac{\partial\bfnc{W}}{\partial \xil{a}}\right),
\end{split}
\end{equation}
where
\begin{equation}\label{eq:Chatij}
\Chatij{l}{a}=\Jdxildxm{l}{m}\sum\limits_{m,j=1}^{3}\Cij{m}{j}\frac{\partial\xil{a}}{\partial x_{j}}.
\end{equation}
The symmetric properties of the viscous flux Jacobians is preserved by the rotation into curvilinear
coordinates: \ie, $\Chatij{l}{a} \:=\: {(\Chatij{a}{l})}\Tr$. See \cite{Fisher2012phd,Parsani2015} for
more details on their construction.
This form of the compressible Navier--Stokes equations, \ie, skew-symmetric form plus the quadratic form of the viscous terms, 
is necessary for the construction of the entropy stable schemes developed in this paper. Note that the geometric conservation 
laws (GCL) are used to obtain the skew-symmetric form from the divergence form of the Navier--Stokes equations:
\begin{equation}\label{eq:GCL}
\sum\limits_{l=1}^{3}\frac{\partial}{\partial\xil{l}}\left(\Jk\frac{\partial\xil{l}}{\partial\xm{m}}\right) = 0,
\quad m = 1,2,3.
\end{equation}

For simplicity, the continuous entropy stability analysis is performed on the Cartesian form 
of the Navier--Stokes equations, given as
\begin{equation}\label{eq:NSC}
\frac{\partial\Q}{\partial t}+\sum\limits_{m=1}^{3}\frac{\partial\bfnc{F}_{\xm{m}}}{\partial\xm{m}}
=\sum\limits_{m,j=1}^{3}\frac{\partial}{\partial\xm{m}}\left(\Cij{m}{j}\frac{\partial\bfnc{W}}{\partial x_{j}}\right).
\end{equation}
Assuming the entropy $\fnc{S}$ is convex (this is guaranteed if $\rho$, $\fnc{T}>0$), then the vector of entropy variables, 
$\bfnc{W}$, simultaneously contracts all the spatial fluxes as follows (see~\cite{FisherCarpenter2013JCPb,Carpenter2014,Parsani2015,Fernandez2019_p_euler} and the references therein for more information):
\begin{equation}\label{eq:contracts}
\frac{\partial \fnc{S}}{\partial \bfnc{Q}}\frac{\partial\Fxm{m}}{\partial \xm{m}}=
\frac{\partial \fnc{S}}{\partial \bfnc{Q}}\frac{\partial\Fxm{m}}{\partial \bfnc{Q}}\frac{\partial\bfnc{Q}}{\partial \xm{m}}=
\frac{\partial\fnc{F}_{\xm{m}}}{\partial\bfnc{Q}}\frac{\partial\bfnc{Q}}{\partial \xm{m}}=
\frac{\partial\fnc{F}_{\xm{m}}}{\partial \xm{m}},\; m=1,2,3,
\end{equation}
where the scalars $\fnc{F}_{\xm{m}}(\bfnc{Q})$ are the entropy fluxes in the $\xm{m}$-direction. 
Therefore, multiplying~\eqref{eq:NSC} by $\bfnc{W}\Tr$, integrating over space gives
{\small
\begin{equation}\label{eq:NSCe1}
\begin{split}
&\displaystyle\int_{\Omega}\left(\bfnc{W}\Tr\frac{\partial\Q}{\partial t}+\sum\limits_{m=1}^{3}\frac{\partial\fnc{F}_{\xm{m}}}{\partial\xm{m}}\right)\mr{d}\Omega
=\displaystyle\int_{\Omega}\bfnc{W}\Tr\left(\sum\limits_{m,j=1}^{3}\frac{\partial}{\partial\xm{m}}\left(\Cij{m}{j}\frac{\partial\bfnc{W}}{\partial x_{j}}\right)\right)\mr{d}\Omega.
\end{split}
\end{equation}
}
The left-hand side of~\eqref{eq:NSCe1} reduces using~\eqref{eq:contracts} and 
$\bfnc{W}\Tr\frac{\partial\Q}{\partial t}=\frac{\partial\fnc{S}}{\partial\bfnc{Q}}\frac{\partial\bfnc{Q}}{\partial t}=\frac{\partial\fnc{S}}{\partial t}$; integration by parts
is used on the right-hand side to obtain 
\begin{equation}\label{eq:NSCe2}
\begin{split}
&\displaystyle\int_{\Omega}\frac{\partial\fnc{S}}{\partial t}\mr{d}\Omega
+\oint_{\Gamma}\sum\limits_{m=1}^{3}\fnc{F}_{\xm{m}}n_{\xm{m}}\mr{d}\Gamma
=\\
&\displaystyle\oint_{\Gamma}\left(\sum\limits_{m,j=1}^{3}\bfnc{W}\Tr\Cij{m}{j}\frac{\partial\bfnc{W}}{\partial\xil{l}}n_{\xm{m}}\right)\mr{d}\Gamma
-\displaystyle\int_{\Omega}\left(\sum\limits_{m,j=1}^{3}\frac{\partial\bfnc{W}\Tr}{\partial\xm{m}}\Cij{m}{j}
\frac{\partial\bfnc{W}}{\partial x_{j}}\right)\mr{d}\Omega.
\end{split}
\end{equation}
Exploiting the symmetries of the $\Cij{m}{j}$ matrices, 
the last term on the right-hand side of~\eqref{eq:NSCe2}, \ie $\displaystyle\int_{\Omega}\left(\sum\limits_{m,j=1}^{3}\frac{\partial\bfnc{W}\Tr}{\partial\xm{m}}\Cij{m}{j}
\frac{\partial\bfnc{W}}{\partial x_{j}}\right)\mr{d}\Omega=\displaystyle\int_{\Omega}\fnc{C}\mr{d}\Omega$, 
can be recast as
\begin{equation}\label{eq:NSdissC}
\fnc{C}\equiv
-
\left[
\frac{\partial\bfnc{W}\Tr}{\partial\xm{1}},
\frac{\partial\bfnc{W}\Tr}{\partial\xm{1}},
\frac{\partial\bfnc{W}\Tr}{\partial\xm{1}}
\right]
\left[
\begin{array}{ccc}
\Cij{1}{1}&\Cij{1}{2}&\Cij{1}{3}\\
\Cij{1}{2}\Tr&\Cij{2}{2}&\Cij{2}{3}\\
\Cij{1}{3}\Tr&\Cij{2}{3}\Tr&\Cij{3}{3}
\end{array}
\right]
\left[
\begin{array}{c}
\frac{\partial\bfnc{W}}{\partial\xm{1}}\\\\
\frac{\partial\bfnc{W}}{\partial\xm{2}}\\\\
\frac{\partial\bfnc{W}}{\partial\xm{3}}
\end{array}
\right].
\end{equation}
The term $\fnc{C}$ is negative semi-definite, therefore~\eqref{eq:NSCe2} reduces to the following inequality:
\begin{equation}\label{eq:NSCe3}
\begin{split}
&\displaystyle\int_{\Omega}\frac{\partial\fnc{S}}{\partial t}\mr{d}\Omega\leq
\oint_{\Gamma}\sum\limits_{m=1}^{3}\left(-\fnc{F}_{\xm{m}}+
\sum\limits_{j=1}^{3}\bfnc{W}\Tr\Cij{m}{j}\frac{\partial\bfnc{W}}{\partial\xil{l}}\right)n_{\xm{m}}\mr{d}\Gamma.
\end{split}
\end{equation}
To obtain a bound on the solution, the inequality~\eqref{eq:NSCe3} is integrated in time and assuming 
nonlinearly well posed boundary conditions and initial condition, and positivity of density and temperature, the result can be turned into a 
bound on the solution in terms of the data of the problem (see, for instance, \cite{dafermos-book-2010,Svard2015}).
\subsection{A $p$-nonconforming algorithm for the compressible Navier--Stokes equations}\label{sec:NSNonConform}
In this section, the $p$-nonconforming algorithm for the convection-diffusion equation is 
applied to the compressible Navier--Stokes equations. Herein, the focus is on discretizing the 
viscous portion of the equations (the details on the inviscid components are 
in~\cite{Fernandez2019_p_euler,Fernandez2018_TM}). 
Recasting the viscous fluxes in terms of entropy variables, and discretizing them using macro element SBP operators
yields the form
\begin{equation}\label{eq:MacroForm}
\sum\limits_{l,a=1}^{3}\frac{\partial}{\partial\xil{l}}\left(\Chatla{l}{a}\frac{\partial\bfnc{W}}{\partial \xil{a}}\right)
 \approx  \sum\limits_{l,a=1}^{3} \Dtildel{l}\matChatla{l}{a}\Dtildel{a}\wtilde. 
\end{equation}
Note that \Eq~\eqref{eq:MacroForm} is precisely the symmetric generalization of the convection-diffusion operator 
to a viscous system;  the stability proof follows immediately (included later in document).  
Greater insight is provided by expressing the macro element
as an element-wise operator, a form that is closer to how a practitioner might implement the algorithm.

The discretization on the $\rmL$ element reads
\begin{equation}\label{eq:NSL}
  \begin{split}
  &\matJk{\rmL}\frac{\mr{d}\qL}{\mr{d}t}+\bm{I}_{ns}^{\rmL}=\sum\limits_{l,a=1}^{3}
  \Dxil{a}^{\rmL}\matChatla{l}{a}^{\rmL}\thetaa{a}^{\rmL}\\
  &-\frac{1}{2}\left(\M^{\rmL}\right)^{-1}\sum\limits_{a=1}^{3}
  \left\{\left(\eNl{\rmL}\left(\eNl{\rmL}\right)\Tr\otimes\PxiloneD{\rmL}\otimes\PxiloneD{\rmL}\otimes\Imat{5}\right)
  \matChatla{1}{a}^{\rmL}\thetaa{a}^{\rmL}\right.\\
  -&\left.\left(\eNl{\rmL}\left(\eonel{\rmH}\right)\Tr\otimes\PxiloneD{\rmL}\IHtoLoneD\otimes\PxiloneD{\rmL}\IHtoLoneD\otimes\Imat{5}\right)
  \matChatla{1}{a}^{\rmH}\thetaa{a}^{\rmH}
  \right\}+\IP^{\rmL},
  \end{split}
\end{equation}
where the interior penalty term, $\IP^{\rmL}$ adds interface dissipation for the viscous terms and will be discussed shortly. 
Moreover, the viscous flux is discretized as
\begin{equation}\label{eq:NSLtheta}
  \begin{split}
  &\thetaa{a}^{\rmL}=\Dxil{a}^{\rmL}\wk{\rmL}-\frac{1}{2}\left(\M^{\rmL}\right)^{-1}
  \left\{\left(\eNl{\rmL}\left(\eNl{\rmL}\right)\Tr\otimes\PxiloneD{\rmL}\otimes\PxiloneD{\rmL}\otimes\Imat{5}\right)
  \wk{\rmL}\right.\\
  -&\left.\left(\eNl{\rmL}\left(\eonel{\rmH}\right)\Tr\otimes\PxiloneD{\rmL}\IHtoLoneD\otimes\PxiloneD{\rmL}\IHtoLoneD\otimes\Imat{5}\right)
  \wk{\rmH}
  \right\}.
  \end{split}
\end{equation}
Similarly, on the $\rmH$ element
\begin{equation}\label{eq:NSH}
  \begin{split}
  &\matJk{\rmH}\frac{\mr{d}\qH}{\mr{d}t}+\bm{I}_{ns}^{\rmH}=\sum\limits_{l,a=1}^{3}
  \Dxil{a}^{\rmH}\matChatla{l}{a}^{\rmH}\thetaa{a}^{\rmH}\\
 &+\frac{1}{2}\left(\M^{\rmH}\right)^{-1}\sum\limits_{a=1}^{3}
  \left\{\left(\eonel{\rmH}\left(\eonel{\rmH}\right)\Tr\otimes\PxiloneD{\rmH}\otimes\PxiloneD{\rmH}\otimes\Imat{5}\right)
  \matChatla{1}{a}^{\rmH}\thetaa{a}^{\rmH}\right.\\
  -&\left.\left(\eonel{\rmH}\left(\eNl{\rmL}\right)\Tr\otimes\PxiloneD{\rmH}\ILtoHoneD\otimes\PxiloneD{\rmH}\ILtoHoneD\otimes\Imat{5}\right)
  \matChatla{1}{a}^{\rmL}\thetaa{a}^{\rmL}
  \right\}+\IP^{\rmH},
  \end{split}
\end{equation}
where 
\begin{equation}\label{eq:NSHtheta}
  \begin{split}
  &\thetaa{a}^{\rmH}=\Dxil{a}^{\rmH}\wk{\rmH}+\frac{1}{2}\left(\M^{\rmH}\right)^{-1}
  \left\{\left(\eonel{\rmH}\left(\eonel{\rmH}\right)\Tr\otimes\PxiloneD{\rmH}\otimes\PxiloneD{\rmH}\otimes\Imat{5}\right)
  \wk{\rmH}\right.\\
  -&\left.\left(\eonel{\rmH}\left(\eNl{\rmL}\right)\Tr\otimes\PxiloneD{\rmH}\ILtoHoneD\otimes\PxiloneD{\rmH}\ILtoHoneD\otimes\Imat{5}\right)
  \wk{\rmL}
  \right\}.
  \end{split}
\end{equation}
The entropy stability properties of the resulting algorithm are given in the next theorem. 
\begin{thrm}
  Assume that the inviscid and IP terms are entropy conservative/stable and that appropriate SATs are utilized. 
  Then the discretization modeled by \Eq~\eqref{eq:NSL} through \Eq~\eqref{eq:NSHtheta} results in an 
  entropy stable discretization.
\end{thrm}
\begin{proof}
  For simplicity, the discretization is recast using the macro element SBP operators and the inviscid and $\IP$ terms are dropped, \ie, 
  \begin{equation}\label{eq:macroNS}
    \matJk{}\frac{\partial\qtilde}{\partial t} =\sum\limits_{l,a=1}^{3}
    \Dtildel{l}\matChatla{l}{a}\Dtildel{a}\wtilde.
  \end{equation}
  Multiplying \Eq~\eqref{eq:macroNS} by $\wtilde\Tr\Mtilde$ gives
  \begin{equation}\label{eq:macroNSent1}
    \wtilde\Tr\Mtilde\matJk{}\frac{\partial\qtilde}{\partial t} =\sum\limits_{l,a=1}^{3}
    \wtilde\Tr\Mtilde\Dtildel{l}\matChatla{l}{a}\Dtildel{a}\wtilde.
  \end{equation}
  The temporal term contracts as follows:
  \begin{equation}\label{eq:temporal}
\wtilde\Tr\Mtilde\matJk{}\frac{\partial\qtilde}{\partial t} = 
\tildeone\Tr\Mtilde\matJk{}\diag\left(\wtilde\right)\frac{\partial\qtilde}{\partial t}=
\barones{}\Tr\Mtilde\matJk{}\frac{\partial\stilde}{\partial t}.
  \end{equation}
Using \Eq~\eqref{eq:temporal} and the SBP property $\Qtildel{l}=-\Qtildel{l}\Tr+\Etildel{l}$ on the spatial terms results in
  \begin{equation}\label{eq:macroNSent2}
    \barones{}\Tr\Mtilde\matJk{}\frac{\partial\stilde}{\partial t}=\sum\limits_{l,a=1}^{3}
    \wtilde\Tr\Etildel{l}\matChatla{l}{a}\Dtildel{a}\wtilde-
    \sum\limits_{l,a=1}^{3}
    \left(\Dtildel{l}\wtilde\right)\Tr\Mtilde\matChatla{l}{a}\Dtildel{a}\wtilde.
  \end{equation}
  The first set of terms on the right-hand side are boundary terms that approximate the surface integrals 
  \begin{equation}
    \sum\limits_{l,a=1}^{3}\oint_{\tilde{\Gamma}}\Chatla{l}{a}\bfnc{W}\Tr
    \frac{\partial\bfnc{W}}{\partial\xil{a}}\nxil{l}\mr{d}\tilde{\Gamma}.
  \end{equation}
  Thus, the scheme telescopes to the boundaries where appropriate SATs need to be imposed to 
  obtain a stability statement. The second set of terms on the right-hand side are negative semi-definite 
  and therefore add dissipation and approximate the following set of volume integrals: 
  \begin{equation*}
    \sum\limits_{l,a=1}^{3}\oint_{\tilde{\Omega}}\Chatla{l}{a}\left(\frac{\partial\bfnc{W}}{\partial\xil{l}}\right)\Tr
    \left(\frac{\partial\bfnc{W}}{\partial\xil{a}}\right)\mr{d}\tilde{\Omega}.
  \end{equation*}
\end{proof}

\begin{remark}
  In contrast to the inviscid terms, the viscous terms are constructed with only one set of metrics. 
  They are used to construct both sets of metrics appearing in the $\matChatla{l}{a}$ matrices. Herein, analytic 
  metrics are used to approximate the viscous terms, although one set of numerical metrics would also suffice.
\end{remark}
\begin{remark}
  The base Euler scheme is element-wise conservative and freestream preserving~\cite{Fernandez2019_p_euler,Fernandez2018_TM}. 
  Element-wise conservation is proven by using the framework of Shi and Shu~\cite{Shi2018} where the scheme is algebraically manipulated into a finite-volume like 
  scheme by discretely integrating the semi-discrete equations element by element. For the viscous contributions to maintain the conservation 
  properties of the base Euler scheme, a unique flux needs to result at the interfaces; this is readily demonstrated by discretely integrating 
  the semi-discrete equations. Similarly, freestream preservation is easily demonstrated by inserting a constant solution 
  into the viscous discretization. 
\end{remark}
\subsubsection{The internal penalty (\IP) terms}\label{sec:NSIP}
The $\IP$ terms are design order zero interface dissipation terms that are constructed to damp neutrally stable ``odd-even'' eigenmodes that arise from 
the LDG viscous operator.  They are cast in terms of the entropy variables so that stability can be proven. 
Specifically, the $\IP$ terms take the following form:
\begin{equation}\label{eq:IPmort}
\begin{split}
\IP^{\rmL}\equiv&-\frac{1}{2}\left(\M^{\rmL}\right)^{-1}\left(\RL\right)\Tr\Porthol{1}^{\rmL}\left(\matJG{\rmL}\right)^{-1}\tilde{\mat{C}}_{1,1}^{\rm{L}}\left(\RL\wk{\rmL}-\IHtoL\RH\wk{\rmH}\right)\\
&-\frac{1}{2}\left(\M^{\rmL}\right)^{-1}\RL\Tr\Porthol{1}^{\rmL}
\IHtoL\left(\matJG{\rmH}\right)^{-1}\tilde{\mat{C}}_{1,1}^{\rmH}\left(\ILtoH\RL\wk{\rmL}-\RH\wk{\rmH}\right),\\\\
\IP^{\rmH}\equiv&-\frac{1}{2}\left(\M^{\rmH}\right)^{-1}\left(\RH\right)\Tr\Porthol{1}^{\rmH}\left(\matJG{\rmH}\right)^{-1}\tilde{\mat{C}}_{1,1}^{\rm{H}}\left(\RH\wk{\rmH}-\ILtoH\RL\wk{\rmL}\right)\\
&-\frac{1}{2}\left(\M^{\rmH}\right)^{-1}\RH\Tr\Porthol{1}^{\rmH}
\ILtoH\left(\matJG{\rmL}\right)^{-1}\tilde{\mat{C}}_{1,1}^{\rmL}\left(\IHtoL\RH\wk{\rmH}-\RL\wk{\rmL}\right),\\
\end{split}
\end{equation}
where 
\begin{equation*}
\begin{split}
\tilde{\mat{C}}_{1,1}^{\rm{L}}\equiv\frac{1}{2}\left(\hat{\mat{C}}_{1,1}\left(\RL\qL\right)+
                                                     \hat{\mat{C}}_{1,1}\left(\IHtoL\RH\qH\right)\right),\qquad \\
\tilde{\mat{C}}_{1,1}^{\rm{H}}\equiv\frac{1}{2}\left(\hat{\mat{C}}_{1,1}\left(\RH\qH\right)+
                                                     \hat{\mat{C}}_{1,1}\left(\ILtoH\RL\qL\right)\right),
\end{split}
\end{equation*}
and the matrices $\matJG{\rmL}$ and $\matJG{\rmH}$ are diagonal matrices with the diagonal entries of $\matJk{\rmL}$ 
and $\matJk{\rmH}$ associated with 
the interface nodes of element $\pL$ and $\pH$, respectively. Moreover, the necessary operators are defined as
\begin{equation*}
  \begin{split}
  &\RL\equiv\eNl{\rmL}\Tr\otimes\Imat{\Nl{\rmL}}\otimes\Imat{\Nl{\rmL}}\otimes\Imat{5},\quad
  \RH\equiv\eonel{\rmH}\Tr\otimes\Imat{\Nl{\rmH}}\otimes\Imat{\Nl{\rmH}}\otimes\Imat{5},\\
  &\Porthol{l}^{\rmL}\equiv\PoneD_{\rmL}\otimes\PoneD_{\rmL}\otimes\Imat{5},\quad
\Porthol{l}^{\rmH}\equiv\PoneD_{\rmH}\otimes\PoneD_{\rmH}\otimes\Imat{5},\\
&\IHtoL\equiv\IHtoLoneD\otimes\IHtoLoneD\otimes\Imat{5},\quad
\ILtoH\equiv\ILtoHoneD\otimes\ILtoHoneD\otimes\Imat{5}.
  \end{split}
\end{equation*}

The stability/dissipativeness of the interface $\bm{I}_{\rm{P}}$ terms is now proven.
\begin{thrm}
The added $\IP$ terms preserves the entropy stable properties of the discretization.
\end{thrm}
\begin{proof}
Contracting the $\IP{}$ terms from element $\pL$ and $\pH$, adding the results 
and rearranging gives
\begin{multline}
\wk{\rmL}\Tr\M^{\rmL}\IP^{\rmL}+\wk{\rmL}\Tr\M^{\rmH}\IP^{\rmH}=\\
-\left(\RL\wk{\rmL}-\IHtoL\RH\wk{\rmH}\right)\Tr\RL\Tr\Porthol{1}^{\rmL}\matJG{\rmL}^{-1}\hat{\mat{C}}_{1,1}^{\rm{L}}
\left(\RL\wk{\rmL}-\IHtoL\RH\wk{\rmH}\right)\\
-\left(\ILtoH\RL\wk{\rmL}-\RH\wk{\rmH}\right)\Tr\Porthol{1}^{rmH}\matJG{\rmH}^{-1}\hat{\mat{C}}_{1,1}^{\rm{H}}
\left(\ILtoH\RL\wk{\rmL}-\RH\wk{\rmH}\right),
\end{multline}
where the SBP preserving properties of the interpolation operators~\eqref{eq:SBPpreserving} have been used. 
The resulting terms are negative semi-definite.
\end{proof} 
\begin{remark}
 The element-wise conservation properties of the scheme are maintained by the IP terms. This is straightforward to demonstrate 
  by discretely integrating the semi-discrete equations. 
  Moreover, the IP terms maintain the freestream preserving properties of the scheme, which can be easily seen by inserting a constant solution 
  into the IP terms and noting that they reduce to zero.
\end{remark}
In Section~\ref{sec:numNS}, two problems are used to characterize the nonconforming algorithms 1) the viscous shock problem, and 2) the Taylor-Green vortex problem. For both, the boundary conditions are weakly imposed by reusing the interface SAT mechanics. For the viscous shock problem, 
the adjoining element's contribution is replaced with the analytical solution. 
\section{Numerical experiments}\label{sec:numNS}
This section presents numerical evidence that the proposed $p$-nonconforming algorithm retains 
the accuracy and robustness of the spatial conforming discretization 
reported in \cite{Carpenter2014,Parsani2015,Carpenter2016,Parsani2016}. We use the unstructured grid solver 
developed at the Extreme Computing Research Center (ECRC) at KAUST. This parallel framework is built on 
top of the Portable and Extensible Toolkit for Scientific computing (PETSc)~\cite{petsc-user-ref}, its mesh topology 
abstraction (DMPLEX)~\cite{KnepleyKarpeev09} and scalable ordinary differential equation (ODE)/differential algebraic equations (DAE) solver library~\cite{abhyankar2018petsc}. 
The systems of ordinary differential equations arising from the spatial
discretizations are integrated using the fourth-order
accurate Dormand--Prince method \cite{dormand_rk_1980} endowed with an adaptive time stepping technique based on digital signal processing \cite{Soderlind2003,Soderlind2006}. To make the temporal error negligible, a tolerance of $10^{-8}$ is always used for the time-step adaptivity. The two-point entropy consistent flux
of Chandrashekar~\cite{Chandrashekar2013} is used for all the test cases.

The errors are computed using volume scaled (for the $L^{1}$ and $L^{2}$ norms) discrete norms as follows:
\begin{equation*}
\begin{split}
  &\|\bm{u}\|_{L^{1}}=\Omega_{c}^{-1}\sum\limits_{\kappa=1}^{K}\ones{\kappa}\Tr\M^{\kappa}\matJk{\kappa}\textrm{abs}\left(\bm{u}_{\kappa}\right),\\
  &\|\bm{u}\|_{L^{2}}^{2}=\Omega_{c}^{-1}\sum\limits_{\kappa=1}^{K}\bm{u}_{\kappa}\M^{\kappa}\matJk{\kappa}\bm{u}_{\kappa},\\
&\|\bm{u}\|_{L^{\infty}}=\max\limits_{\kappa=1\dots K}\textrm{abs}\left(\bm{u}_{\kappa}\right),
\end{split}
\end{equation*}
where $\Omega_{c}$ is the volume of $\Omega$ computed as 
$\Omega_{c}\equiv\sum\limits_{\kappa=1}^{K}\ones{\kappa}\Tr\M^{\kappa}\matJk{\kappa}\ones{\kappa}$.

\subsection{Viscous shock propagation}
For verification and characterization of the full compressible Navier--Stokes algorithm, the propagation of a viscous 
shock is used. For this test case an exact time-dependent solution is known, under the assumption of a Prandtl number of 
$Pr=3/4$. The momentum $\fnc{V}(x_1)$ satisfies the ODE
\begin{equation}
\begin{split}
&\alpha\fnc{V}\frac{\partial\fnc{V}}{\partial\xm{1}}-(\fnc{V}-1)(\fnc{V}-\fnc{V}_{f})=0;\;-\infty\leq\xm{1}\leq+\infty,\; t\ge 0,\\
\end{split}
\end{equation}
whose solution can be written implicitly as
{\small
\begin{equation}\label{eq:implicit_sol_vs}
  \begin{split}
&\xm{1}-\frac{1}{2}\alpha\left(\log\left|(\fnc{V}(x_1)-1)(\fnc{V}(x_1)-\fnc{V}_{f})\right|\right.\\
&\left.+\frac{1+\fnc{V}_{f}}{1-\fnc{V}_{f}}\log\left|\frac{\fnc{V}(x_1)-1}{\fnc{V}(x_1)-\fnc{V}_{f}}\right|\right) = 0,
  \end{split}
\end{equation}
}
where
\begin{equation}
\fnc{V}_{f}\equiv\frac{\fnc{U}_{L}}{\fnc{U}_{R}},\qquad
\alpha\equiv\frac{2\gamma}{\gamma + 1}\frac{\,\mu}{Pr\dot{\fnc{M}}}.
\end{equation}
Here $\fnc{U}_{L/R}$ are known velocities to the left and right of the shock at
$-\infty$ and $+\infty$, respectively, $\dot{\fnc{M}}$ is the constant mass
flow across the shock, $Pr$ is the Prandtl number, and $\mu$ is the dynamic
viscosity.
The mass and total enthalpy are constant across the shock. Moreover,
the momentum and energy equations become redundant.

For our tests, $\fnc{V}$ is computed from Equation \eqref{eq:implicit_sol_vs}
to machine precision using bisection.
The moving shock solution is obtained by applying a uniform translation to the above solution.
The shock is located at the center of the domain at $t=0$ and the following
values are used: $M_{\infty}=2.5$, $Re_{\infty}=10$, and $\gamma=1.4$.
The domain is given by 
\[
	\xm{1}\in[-0.5,0.5], \;\xm{2}\in[-0.5,0.5],\;\xm{3}\in[-0.5,0.5],\; t\in[0,0.5].
\]

A grid convergence study is performed to investigate the order of convergence
of the nonconforming algorithm. We use a sequence of nested grids 
obtained from a base grid (i.e., the coarsest grid) which is constructed as follows:
\begin{itemize}
\item Divide the computational domain with four hexahedral elements in each coordinate direction.
\item Assign the solution polynomial degree in each element to a random integer chosen uniformly from the set $\{p_s, p_s+1\}$.
\item Approximate with a $p_s$th-order polynomial the element interfaces.
\item Perturb the nodes that are used to define the $p_s$th-order polynomial
  approximation of the element interfaces as follows:
\begin{equation*}
\begin{split}
&x_1 = x_{1,*} + \frac{1}{15} L_1 \cos \left(   a \right) \cos \left( 3 b \right) \sin \left( 4 c \right), \\
&x_2 = x_{2,*} + \frac{1}{15} L_2 \sin \left( 4 a \right) \cos \left(   b \right) \cos \left( 3 c \right),\\
&x_3 = x_{3,*} + \frac{1}{15} L_3 \cos \left( 3 a \right) \sin \left( 4 b \right) \cos \left(   c \right),
\end{split}
\end{equation*}
where
\begin{equation*}
\begin{split}
  &a = \frac{\pi}{L_1} \left(x_{1,*}-\frac{x_{1,H}+x_{1,L}}{2}\right), \\
  &b = \frac{\pi}{L_2} \left(x_{2,*}-\frac{x_{2,H}+x_{2,L}}{2}\right),\\
  &c = \frac{\pi}{L_3} \left(x_{3,*}-\frac{x_{3,H}+x_{3,L}}{2}\right).
\end{split}
\end{equation*}
The lengths $L_1$, $L_2$ and $L_3$ are the dimensions of the computational domain in the three
coordinate directions and the sub-script $*$ is the unperturbed coordinates of
the nodes. This step yields a ``perturbed'' $p_s$th-order interface polynomial representation.
\item Compute the coordinate of the LGL points at the element interface
by evaluating the ``perturbed'' $p_s$th-order polynomial at the tensor-product LGL points used 
    to define the cell solution polynomial of order $p_s$ or $p_s+1$.
\end{itemize}
The element interfaces are perturbed as described above to test the conservation of entropy and therefore the 
freestream condition.\footnote{In a general setting, element interfaces can also 
be boundary element interfaces.}

\onecolumn
\begin{table}[htbp!]
\vspace{-0.75cm}
\begin{center}
\scriptsize
\begin{tabular}{||l||c|c|c|c|c|c||c|c|c|c|c|c||}
\hline \hline
        & \multicolumn{6}{c||}{Conforming, $p=1$} & \multicolumn{6}{c||}{Invisible mortar, $p=1$ and $p=2$}\\ \hline
 Grid & $L^1$    & Rate  & $L^2$    & Rate  & $L^{\infty}$ & Rate  & $L^1$    & Rate  & $L^2$    & Rate  & $L^{\infty}$ & Rate  \\ \hline
 4    & 5.43E-02 & -     & 6.54E-02 & -     & 1.40E-01   & -       & 4.48E-02 & -     & 5.99E-02 & -     & 2.16E-01     & -     \\ \hline
 8    & 2.04E-02 & -1.41 & 2.92E-02 & -1.16 & 8.42E-02   & -0.73   & 1.39E-02 & -1.69 & 2.21E-02 & -1.44 & 7.01E-02     & -1.62 \\ \hline
 16   & 5.56E-03 & -1.87 & 8.45E-03 & -1.79 & 2.85E-02   & -1.57   & 3.39E-03 & -2.04 & 5.94E-03 & -1.89 & 1.95E-02     & -1.85 \\ \hline
 32   & 1.44E-03 & -1.94 & 2.23E-03 & -1.92 & 8.12E-03   & -1.81   & 8.74E-04 & -1.96 & 1.50E-03 & -1.99 & 5.24E-03     & -1.89 \\ \hline
 64   & 3.68E-04 & -1.97 & 5.66E-04 & -1.98 & 2.26E-03   & -1.84   & 2.12E-04 & -2.04 & 3.70E-04 & -2.02 & 1.34E-03     & -1.97 \\ \hline
 128  & 9.28E-05 & -1.99 & 1.43E-04 & -1.99 & 6.05E-04   & -1.90   & 5.21E-05 & -2.03 & 9.21E-05 & -2.01  & 3.74E-04     & -1.84 \\ \hline
\end{tabular}
\caption{Convergence study of the viscous shock propagation: $p=1$ with $p=2$; density error.}.
\label{tab:iv_p1p2}
\end{center}
\end{table}

\begin{table}[htbp!]
\vspace{-0.75cm}
\begin{center}
\scriptsize
\begin{tabular}{||l||c|c|c|c|c|c||c|c|c|c|c|c||}
\hline \hline
      & \multicolumn{6}{c||}{Conforming, $p=2$} & \multicolumn{6}{c||}{Invisible mortar, $p=2$ and $p=3$}\\ \hline
 Grid & $L^1$    & Rate  & $L^2$    & Rate  & $L^{\infty}$ & Rate  & $L^1$    & Rate  & $L^2$    & Rate  & $L^{\infty}$ & Rate  \\ \hline
 4    & 1.78E-02 & -     & 2.68E-02 & -     & 1.36E-01     & -     & 1.26E-02 & -     & 2.28E-02 & -     & 1.41E-01     & -     \\ \hline
 8    & 2.93E-03 & -2.60 & 5.05E-03 & -2.41 & 5.98E-02     & -1.19 & 2.00E-03 & -2.65 & 4.08E-03 & -2.48 & 4.88E-02     & -1.53 \\ \hline
 16   & 3.86E-04 & -2.92 & 6.93E-04 & -2.87 & 1.09E-02     & -2.45 & 2.85E-04 & -2.81 & 5.90E-04 & -2.79 & 9.85E-03     & -2.31 \\ \hline
 32   & 5.55E-05 & -2.80 & 1.03E-04 & -2.74 & 2.23E-03     & -2.29 & 4.40E-05 & -2.70 & 9.28E-05 & -2.67 & 1.66E-03     & -2.57 \\ \hline
 64   & 8.96E-06 & -2.63 & 1.79E-05 & -2.53 & 4.96E-04     & -2.17 & 7.47E-06 & -2.56 & 1.57E-05 & -2.56 & 4.30E-04     & -1.95 \\ \hline
 128  & 1.46E-06 & -2.66 & 2.99E-06 & -2.58 & 8.96E-05     & -2.47 & 1.20E-06 & -2.64 & 2.54E-06 & -2.63 & 8.20E-05     & -2.39  \\ \hline
\end{tabular}
  \caption{Convergence study of the viscous shock propagation: $p=2$ with $p=3$; density error.}.
\label{tab:iv_p2p3}
\end{center}
\end{table}

\begin{table}[htbp!]
\vspace{-0.75cm}
\begin{center}
\scriptsize
\begin{tabular}{||l||c|c|c|c|c|c||c|c|c|c|c|c||}
\hline \hline
      & \multicolumn{6}{c||}{Conforming, $p=3$} & \multicolumn{6}{c||}{Invisible mortar, $p=3$ and $p=4$}\\ \hline
 Grid & $L^1$    & Rate  & $L^2$    & Rate  & $L^{\infty}$ & Rate  & $L^1$    & Rate  & $L^2$    & Rate  & $L^{\infty}$ & Rate  \\ \hline
 4    & 4.45E-03 & -     & 7.52E-03 & -     & 7.51E-02     & -     & 3.19E-03 & -     & 6.22E-03 & -     & 6.34E-02     & -     \\ \hline
 8    & 3.40E-04 & -3.71 & 6.50E-04 & -3.53 & 1.19E-02     & -2.66 & 2.57E-04 & -3.63 & 5.41E-04 & -3.52 & 1.12E-02     & -2.50 \\ \hline
 16   & 2.67E-05 & -3.67 & 5.36E-05 & -3.60 & 1.20E-03     & -3.30 & 2.09E-05 & -3.62 & 4.61E-05 & -3.55 & 1.05E-03     & -3.41 \\ \hline
 32   & 1.95E-06 & -3.77 & 4.25E-06 & -3.66 & 1.25E-04     & -3.26 & 1.64E-06 & -3.67 & 3.81E-06 & -3.60 & 9.23E-05     & -3.51 \\ \hline
 64   & 1.48E-07 & -3.72 & 3.67E-07 & -3.53 & 1.12E-05     & -3.48 & 1.20E-07 & -3.77 & 2.91E-07 & -3.71 & 1.00E-05     & -3.21  \\ \hline
\end{tabular}
  \caption{Convergence study of the viscous shock propagation: $p=3$ with $p=4$; density error.}.
\label{tab:iv_p3p4}
\end{center}
\end{table}

\begin{table}[htbp!]
\vspace{-0.75cm}
\begin{center}
\scriptsize
\begin{tabular}{||l||c|c|c|c|c|c||c|c|c|c|c|c||}
\hline \hline
      & \multicolumn{6}{c||}{Conforming, $p=4$} & \multicolumn{6}{c||}{Invisible mortar, $p=4$ and $p=5$}\\ \hline
 Grid & $L^1$    & Rate  & $L^2$    & Rate  & $L^{\infty}$ & Rate  & $L^1$     & Rate  & $L^2$    & Rate  & $L^{\infty}$ & Rate   \\ \hline
 4    & 1.21E-03 & -     & 2.28E-03 & -     & 2.50E-02     & -     & 1.11E-03  & -     & 2.18E-03 & -     & 2.99E-02     & -      \\ \hline
 8    & 8.50E-05 & -3.83 & 1.54E-04 & -3.88 & 3.04E-03     & -3.04 & 6.17E-05  & -4.16 & 1.21E-04 & -4.16 & 2.25E-03     & -3.73  \\ \hline
 16   & 2.75E-05 & -4.95 & 5.66E-06 & -4.77 & 1.52E-04     & -4.32 & 2.14E-06  & -4.85 & 4.88E-06 & -4.64 & 1.51E-04     & -3.90  \\ \hline
 32   & 1.16E-07 & -4.57 & 2.54E-07 & -4.48 & 7.54E-06     & -4.33 & 8.93E-08  & -4.58 & 2.18E-07 & -4.48 & 6.89E-06     & -4.45  \\ \hline
 64   & 5.21E-09 & -4.48 & 1.11E-08 & -4.52 & 4.01E-07     & -4.23 & 3.53E-09  & -4.66 & 9.00E-09 & -4.60 & 2.79E-07     & -4.63  \\ \hline
\end{tabular}
  \caption{Convergence study of the viscous shock propagation: $p=4$ with $p=5$; density error.}.
\label{tab:iv_p4p5}
\end{center}
\end{table}

\begin{table}[htbp!]
\vspace{-0.75cm}
\begin{center}
\scriptsize
\begin{tabular}{||l||c|c|c|c|c|c||c|c|c|c|c|c||}
\hline \hline
      & \multicolumn{6}{c||}{Conforming, $p=4$} & \multicolumn{6}{c||}{Invisible mortar, $p=4$ and $p=5$}\\ \hline
 Grid & $L^1$    & Rate  & $L^2$    & Rate  & $L^{\infty}$ & Rate  & $L^1$     & Rate  & $L^2$    & Rate  & $L^{\infty}$ & Rate  \\ \hline
 4    & 5.14E-04 & -     & 8.81E-04 & -     & 1.14E-02     & -     & 3.87E-04  & -     & 7.30E-04 & -     & 1.07E-02     & -     \\ \hline
 8    & 1.07E-05 & -5.59 & 2.21E-05 & -5.32 & 4.95E-04     & -4.52 & 8.12E-06  & -5.57 & 1.83E-05 & -5.32 & 4.57E-04     & -4.55 \\ \hline
 16   & 1.91E-07 & -5.81 & 4.10E-07 & -5.75 & 1.11E-05     & -5.47 & 1.44E-07  & -5.82 & 3.39E-07 & -5.76 & 1.08E-05     & -5.41 \\ \hline
 32   & 3.59E-09 & -5.73 & 8.71E-09 & -5.55 & 2.88E-07     & -5.27 & 2.85E-09  & -5.66 & 7.25E-09 & -5.55 & 2.32E-07     & -5.54 \\ \hline
\end{tabular}
  \caption{Convergence study of the viscous shock propagation: $p=5$ with $p=6$; density error.}.
\label{tab:iv_p5p6}
\end{center}
\end{table}

\begin{table}[htbp!]
\vspace{-0.75cm}
\begin{center}
\scriptsize
\begin{tabular}{||l||c|c|c|c|c|c||c|c|c|c|c|c||}
\hline \hline
      & \multicolumn{6}{c||}{Conforming, $p=4$} & \multicolumn{6}{c||}{Invisible mortar, $p=4$ and $p=5$}\\ \hline
 Grid & $L^1$    & Rate  & $L^2$    & Rate  & $L^{\infty}$ & Rate  & $L^1$    & Rate  & $L^2$    & Rate  & $L^{\infty}$ & Rate  \\ \hline
 4    & 1.17E-04 & -     & 2.31E-04 & -     & 3.18E-03     & -     & 9.95E-05 & -     & 2.09E-04 & -     & 3.35E-03     & -     \\ \hline
 8    & 1.66E-06 & -6.13 & 3.18E-06 & -6.18 & 9.00E-05     & -5.15 & 1.27E-06 & -6.29 & 2.61E-06 & -6.32 & 4.98E-05     & -6.07 \\ \hline
 16   & 1.50E-08 & -6.80 & 3.43E-08 & -6.54 & 9.62E-07     & -6.55 & 1.17E-08 & -6.77 & 2.88E-08 & -6.50 & 9.54E-07     & -5.71 \\ \hline
 32   & 1.51E-10 & -6.63 & 3.72E-10 & -6.53 & 1.31E-08     & -6.20 & 1.10E-10 & -6.73 & 2.98E-10 & -6.60 & 1.64E-08     & -5.86 \\ \hline
\end{tabular}
  \caption{Convergence study of the viscous shock propagation: $p=6$ with $p=7$; density error.}.
\label{tab:iv_p6p7}
\end{center}
\end{table}
\twocolumn
For all the degree tested (i.e. $p=1$ to $p=7$), the order of convergence of the conforming and nonconforming algorithms is very close to each other. However, note for both the $L^{1}$ and 
$L^{2}$ norms the nonconforming algorithm is more accurate than the conforming one. 
In the discrete $L^{\infty}$ norm instead, the nonconforming scheme 
is sometimes slightly worse than the conforming scheme; this results from the interpolation matrices being sub-optimal at nonconforming interfaces.

\subsection{Taylor-Green vortex at $Re=1,600$}
The purpose of this section is to demonstrate that the nonconforming algorithm has the same stability properties as the 
conforming algorithm. To do so, the Taylor--Green vortex problem is solved, which is a flow that degenerates to 
turbulence over time; therefore its solution is representative of the behavior of the algorithm for the solution 
of under-resolved turbulent flows. 

The Taylor--Green vortex case is solved on a periodic cube $[-\pi L\leq x,y,z\leq \pi L]$, where the initial condition 
is given by
\begin{equation}\label{eq:TG}
\begin{split}
&\fnc{U}_{1} = \fnc{V}_{0}\sin\left(\frac{x_{1}}{L}\right)\cos\left(\frac{x_{2}}{L}\right)\cos\left(\frac{x_{3}}{L}\right),\\
&\fnc{U}_{2} = -\fnc{V}_{0}\cos\left(\frac{x_{1}}{L}\right)\sin\left(\frac{x_{2}}{L}\right)\cos\left(\frac{x_{3}}{L}\right),\\
&\fnc{U}_{3} = 0,\\
&\fnc{P} = \fnc{P}_{0}+\frac{\rho_{0}\fnc{V}_{0}^{2}}{16}\left[\cos\left(\frac{2x_{1}}{L}+\cos\left(\frac{2x_{2}}{L}\right)\right)\right]
\left[\cos\left(\frac{2x_{3}}{L}+2\right)\right].
\end{split}
\end{equation}
The flow is initialized to be isothermal, \ie, $\fnc{P}/\rho=\fnc{P}_{0}/\rho_{0}=R\fnc{T}_{0}$, and $\fnc{P}_{0}=1$, $\fnc{T}_{0}=1$, $L=1$, and $\fnc{V}_{0} = 1$. 
Finally, the Reynolds number is defined by $Re=(\rho_{0}\fnc{V}_{0})/\mu$, where $\mu$ is the dynamic viscosity.  

The solver that is used is implemented in a compressible fluid dynamics code. Therefore, to obtain results that are reasonably close to those found for the incompressible Navier--Stokes
equations, a Mach number of $M = 0.05$ is used. The Reynolds number is set to $Re=1,600$, and $\rho_{0}=\gamma M^{2}$, where $\gamma=1.4$. The Prandtl number is set to $Pr=0.71$. A perturbed grid with eight hexahedrons
elements in each coordinate direction is used. This grid is constructed by perturbing the element 
interfaces as previously described, with the exception that
the interfaces are approximated using the minimum solution 
polynomial degree set in the simulation.
All the computations are performed without additional stabilization mechanisms (dissipation model, filtering, \etc), where the only numerical dissipation originates from the upwind inter-element coupling procedure.

Figure~\ref{fig:dkedt_tg} shows the time rate of change of the kinetic energy, $dke/dt$, for
the non conforming algorithm using a random distribution of solution polynomial
order between i) $p=2$ and $p=12$, ii) $p=7$ and $p=13$, and iii) $p=13$ and $p=15$. The 
incompressible DNS solution reported in \cite{de_wiart_tgv} is plotted as a reference. 
All simulations are stable on a variety of meshes with poor quality, which provides
numerical evidence that the nonconforming 
scheme inherits the same stability characteristics of the conforming 
algorithm \cite{Carpenter2014,Parsani2015,Carpenter2016b,Carpenter2016}. 

\begin{figure}[!t]
   \centering
   \includegraphics[width=0.5\textwidth]{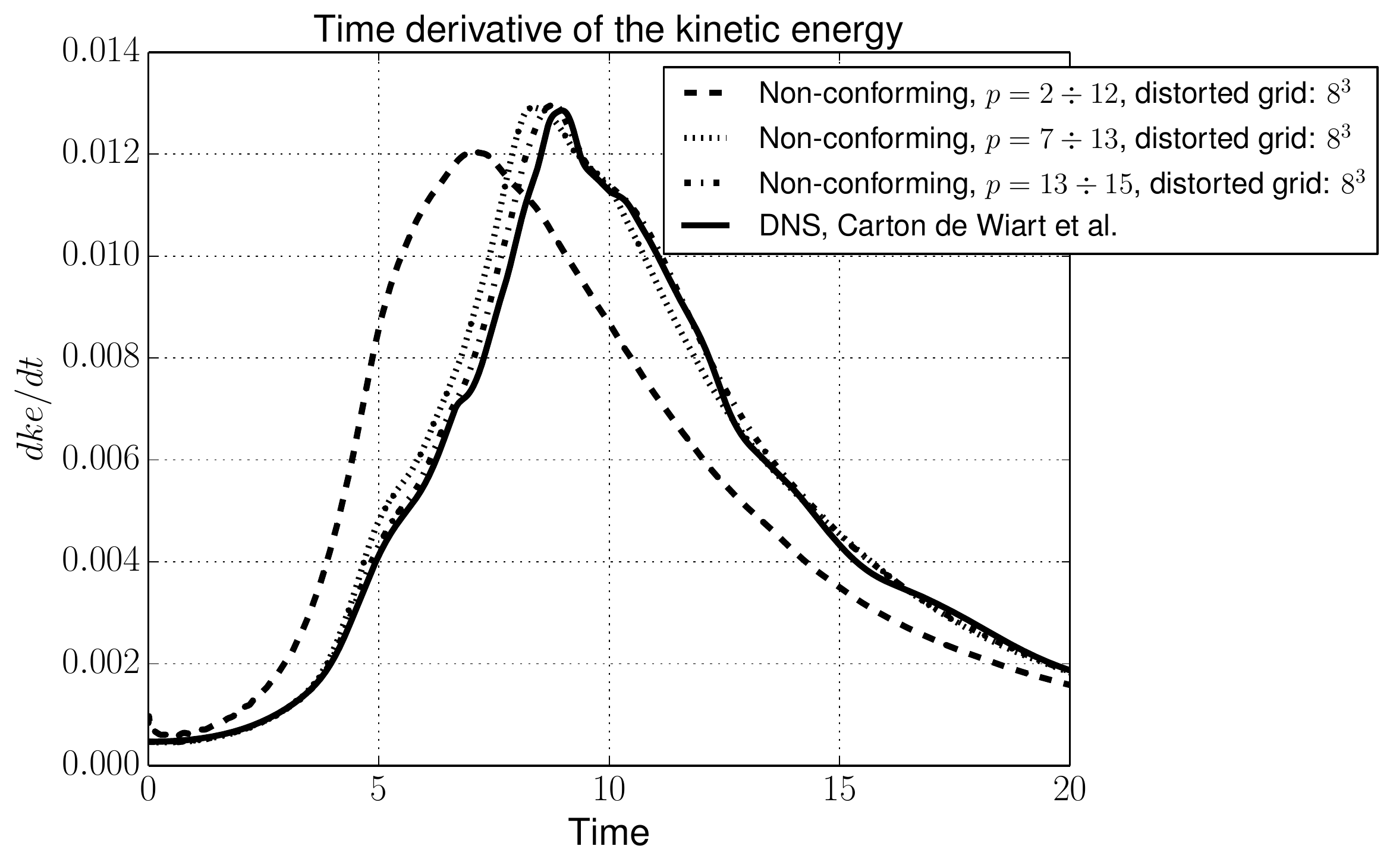}
   \caption{Evolution of the time derivative of the kinetic energy 
   for the Taylor--Green vortex at $Re = 1,600$, $M = 0.05$.}
   \label{fig:dkedt_tg}
\end{figure}

\subsection{Flow around a sphere at $Re=2,000$}
In this section, we test our implementation within a more complicated setting represented by the flow around a sphere at $Re=2,000$ and $M=0.05$. With this value of the Reynolds number the flow is fully turbulent.
In this case, a sphere of diameter $d$ is centered at the origin of the axes, and a box is respectively extended $20d$ and $60d$ upstream
and downstream the direction of the flow; the box size is $30d$ in both the $x_2$ and $x_3$ directions. As boundary conditions, 
we consider adiabatic solid walls at the surface of the sphere \cite{dalcin_2019_wall_bc} and far field on all faces of the box.
The surface of the sphere is first triangulated using third order simplices, and a boundary
layer composed of triangular prisms is extruded from the sphere surface for a total length of $3d$. The rest of the domain is meshed with an unstructured tetrahedral mesh. We then obtain an unstructured conforming hexahedral mesh by uniformly splitting each tetrahedron in four hexahedra, and each prism in three hexahedra, resulting in a total of 22,648 hexahedral elements. Figure 
\ref{fig:mesh-sphere} shows a zoom of the grid near the sphere. The colors indicate 
the solution polynomial degree used in each cell.
The quality of the elements is good in the boundary layer region whereas in the other portion of the domain is fairly poor. This
choice is intentional to demonstrate the performance of the algorithm on non-ideal grids.

\begin{figure}[!t]
   \centering
   \includegraphics[width=0.5\textwidth]{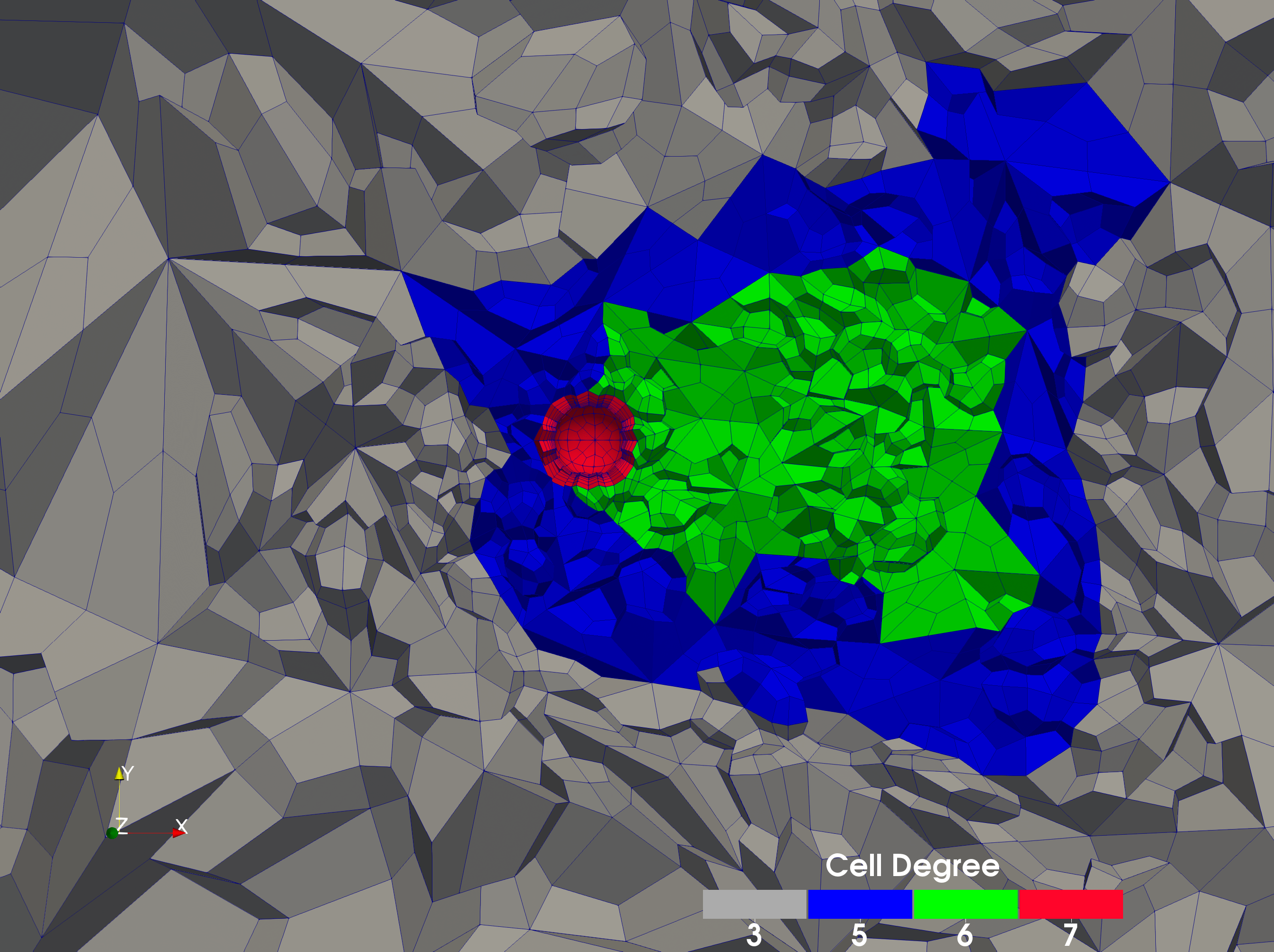}
   \caption{Polynomial order distribution for the mesh around a sphere at $Re = 2,000$, $M = 0.05$.}
   \label{fig:mesh-sphere}
\end{figure}

We compute the time-average value of the drag coefficient, $\langle C_D \rangle$, 
and we compare it with the value reported in in literature \cite{Munson_1990}.
Figure \ref{fig:cd-sphere}, shows a time window of the evolution of the drag coefficient. 
To compute $\langle C_D \rangle$ we average the flow field and hence the aerodynamic forces for 
200 time units. From Table \ref{tab:sphere_cd}, it can be seen that the computed time-average drag 
coefficient matches very well the value reported in literature. 

\begin{figure}[!t]
   \centering
   \includegraphics[width=0.5\textwidth]{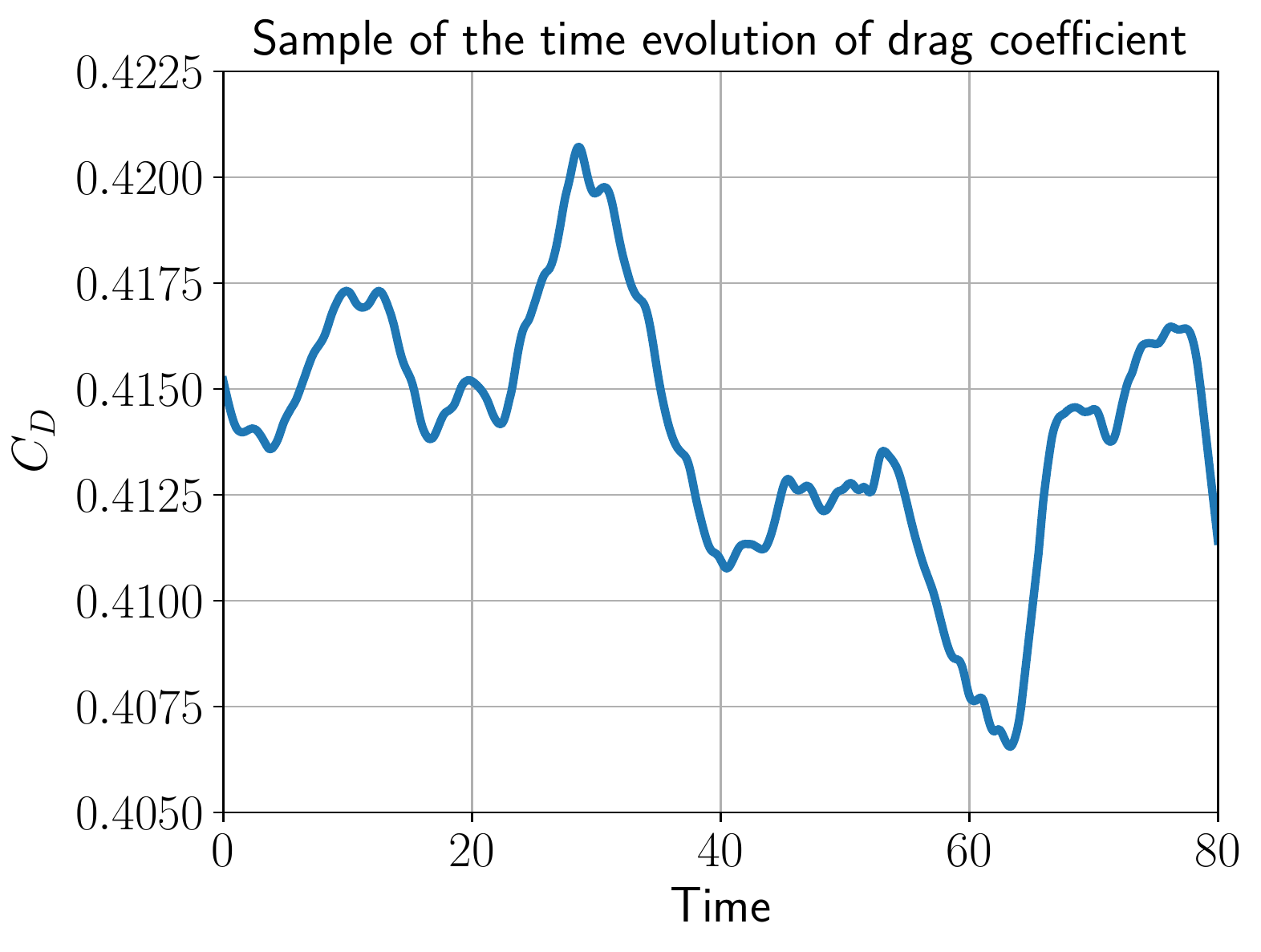}
   \caption{Sample of the evolution of the drag coefficient around a sphere at $Re = 2,000$, $M = 0.05$.}
   \label{fig:cd-sphere}
\end{figure}

\begin{table}[htbp!]
   \centering
\begin{tabular}{||c|c||}
\hline \hline
                                 & $\langle C_D \rangle$ \\ \hline 
Munson et al. \cite{Munson_1990} & 0.412 \\ \hline
Present                          & 0.414 \\ \hline
\end{tabular}
\caption{Time-average drag coefficient of a sphere at $Re = 2,000$, $M = 0.05$.}.
\label{tab:sphere_cd}
\end{table}

\section{Conclusions}\label{sec:conclusions}
In this paper, the entropy conservative $p-$refinement/coarsening nonconforming algorithm in~\cite{Fernandez2019_p_euler,Fernandez2018_TM} is extended 
to the compressible Navier--Stokes equations. The viscous terms are transformed into a quadratic form, 
in terms of the entropy variables, so that entropy conservation/stability of the original compressible Euler code is maintained. 
An LDG-IP type approach is used in discretizing the viscous terms and entropy stability of the scheme is proven. The accuracy and stability characteristics of the resulting numerical 
schemes are shown to be comparable to those of the original conforming scheme, in the context of the viscous shock problem, the 
Taylor--Green Vortex problem, and a turbulent flow past a sphere.

\section*{Acknowledgments}
Special thanks are extended to Dr. Mujeeb R. Malik for partially funding this work as part of
NASA's ``Transformational Tools and Technologies'' ($T^3$) project.  The research reported in this publication was
also supported by funds from King Abdullah University of Science and Technology (KAUST).
We are thankful for the computing resources of the Supercomputing Laboratory and the Extreme Computing Research Center at KAUST. 
Gregor Gassner and Lucas Friedrich has been supported by the European Research Council (ERC) under the European Union’s Eights Framework Program Horizon 2020 with the research project Extreme, ERC grant agreement no. 714487.








\bibliographystyle{aiaa}
\bibliography{Bib}
\end{document}